\documentclass[11pt,a4paper]{article}
\usepackage[utf8]{inputenc}
\usepackage{amsmath}
\usepackage{amsfonts}
\usepackage{amssymb}
\usepackage{graphicx}
\usepackage{color}
\usepackage[left=2cm,right=2cm,top=2cm,bottom=2cm]{geometry}
\usepackage{placeins}
\usepackage{amsthm}
\usepackage[makeroom]{cancel}
\RequirePackage[numbers]{natbib}
\usepackage{hyperref}
\usepackage{ulem}
\usepackage{stackrel}
\usepackage{dsfont}

\usepackage{pgffor}
\usepackage{multicol,multirow}
\usepackage{algorithm}
\usepackage[noend]{algpseudocode}
\makeatletter
\def\BState{\State\hskip-\ALG@thistlm}

\newtheorem{theorem}{Theorem}
\newtheorem{definition}{Definition}
\newtheorem{proposition}{Proposition}
\newtheorem{remark}{Remark}
\newtheorem{condition}{Condition}
\newtheorem{lemma}{Lemma}

\newcommand{\R}{\mathbb{R}}

\newcommand{\addvar}[1]{\, + \,}

\author{}

\newcommand{\Tr}{\operatorname{Tr}}

\newcommand{\IN}{\mathbb{N}}
\newcommand{\IR}{\mathbb{R}}
\newcommand{\E}{\mathbb{E}}
\newcommand{\argmin}[1]{\underset{#1}{\mathrm{argmin}}~}
\newcommand{\Esp}{\mathbb{E}}
\newcommand{\Var}{\mathrm{Var}}
\newcommand{\Int}[1]{\mathring{#1}}
\newcommand{\ITheta}{\Int{\Theta}}
\newcommand{\Cov}{\mathrm{Cov}}
\newcommand{\cW}{\mathcal{W}}

\definecolor{green}{rgb}{0, 0.5, 0}

\begin{document}

\title{Bounds in $L^1$ Wasserstein distance on the normal approximation of general M-estimators}

\author{Fran\c{c}ois Bachoc \\
	Institut de Mathématiques de Toulouse; \\ UMR5219. Université de Toulouse; \\
CNRS. UT3, F-31062 Toulouse, France \\
~ \\
	 Max Fathi \\
	 Laboratoire Jacques Louis Lions \& Laboratoire de Probabilit\'es Statistique et Mod\'elisation; \\ Université de Paris, France
 }

\maketitle

\abstract{We derive quantitative bounds on the rate of convergence in $L^1$ Wasserstein distance of general M-estimators, with an almost sharp (up to a logarithmic term) behavior in the number of observations. We focus on situations where the estimator does not have an explicit expression as a function of the data. The general method may be applied even in situations where the observations are not independent. Our main application is a rate of convergence for cross validation estimation of covariance parameters of Gaussian processes.} \\

\vspace{0.1cm}

\section{Introduction}

Our goal here is to derive quantitative bounds for approximate normality of parameter estimators that arise as minimizers of certain random functions. The main example to keep in mind is maximum likelihood estimation \cite[Chapter 5.5]{van2000asymptotic}, but other problems fit in the framework we shall consider, including least square estimators \cite{pronzato2013design} and cross validation \cite{Bachoc2013cross,Zhang2010kriging}. 

Consider a fixed compact parameter space $\Theta \subset \IR^p$ and a sequence of random functions $(M_n)_{n \in \IN}$, where for $n \in \IN$, $M_n: \Theta \to \IR$. 
Throughout, $\IN$ is the set of non-zero natural numbers.
The variable $n$ should be thought of as a sample size, and $M_n$ the function for which a  minimizer will be the M-estimator of interest, which is a (measurable) random vector $\hat{\theta}_n \in \Theta$ such that
\begin{equation} \label{eq:hat:theta:general}
	\hat{\theta}_n \in \argmin{\theta \in \Theta} M_n(\theta).
\end{equation}

A classical family of M-estimators is given by functions of the form
\begin{equation} \label{eq:M:n:generic:average}
M_n(\theta) = \frac{1}{n}\sum_{i=1}^n \rho(\theta, X_i)
\end{equation}
where the $X_i$ are the sample independent data, valued in a space $\mathcal{X}$, and $\rho : \Theta \times \mathcal{X} \to \IR$ is a fixed function. We shall address in details this class in Sections \ref{subsection:average:iid} and \ref{subsection:logistic}, but investigation shall go beyond this framework, in particular to cover covariance estimation for Gaussian processes, addressed in Section \ref{subsection:cross:validation}. 

Our goal will be to derive quantitative central limit theorems in $L^1$ Wasserstein (or optimal transport) distance for the fluctuations of $\hat{\theta}_n $ around a deterministic parameter $\theta_{0,n}$ (that is allowed to depend on $n$). The simplest example is when $\theta_{0,n} = \theta_0$ is fixed, typically when $M_n$ stems from the likelihood function and there is a fixed data generating process characterized by the ``true'' parameter $\theta_0$ \cite[Chapter 5.5]{van2000asymptotic}. Nevertheless, we allow for a sample-size dependent $\theta_{0,n}$ which enables to address relevant situations such as misspecified models \cite{bachoc2020uniformly,berk2013valid,huber1967under,white1982maximum}. In particular, in  \cite{bachoc2020uniformly,berk2013valid}, the parameter of interest $\theta_{0,n}$ that $\hat{\theta}_n $ estimates explicitly depends on sample size. 

In the context of this paper, it is typically already known that the distribution of $n^{1/2} (   \hat{\theta}_n  - \theta_{0,n} )$ converges to a Gaussian distribution. 
General techniques for showing this convergence are available in a wealth of contributions, see for instance \cite{casella2021statistical,potscher2013dynamic,van2000asymptotic} and references therein. 
Our goal is then to go beyond the convergence between these two distributions (for which, usually, no rates are available)
by providing quantitative bounds on their $L^1$ Wasserstein distance. 
In this view, the main challenge is the M-estimation setting, which often entails that no explicit expression of $\hat{\theta}_n$ is available.   
Our main abstract result, Theorem \ref{theorem:general:bound}, is a general statement about reducing the problem to a central limit theorem for an explicit function of the data. More precisely, the $L^1$ Wasserstein distance between the distribution of $n^{1/2} (  \hat{\theta}_n  - \theta_{0,n} )$ and a Gaussian distribution is bounded by the sum of a term of order $n^{-1/2}$ (up to a log factor) and the distance between a Gaussian distribution and the normalized gradient of $M_n$ at $\theta_{0,n}$. 

Hence, Theorem \ref{theorem:general:bound} enables to reduce the problem to quantifying the asymptotic normality of 
this normalized gradient. Since this quantity is explicit, there are many techniques in the literature that can be applied. We shall discuss this aspect of the problem in Section \ref{section:background:quantitative:CLT}. 

We shall illustrate the benefits of Theorem \ref{theorem:general:bound} with several examples of functions $M_n$: averages of independent functions in Section  \ref{subsection:average:iid}, maximum likelihood for logistic regression in Section \ref{subsection:logistic} and cross validation estimation of covariance parameters of Gaussian processes in Section \ref{subsection:cross:validation}. This last example highlights the flexibility of our techniques, since the observations are dependent and the function $M_n$ is not based on the likelihood. In all these three cases, eventually, we provide a bound, for the $L^1$ Wasserstein distance between the distribution of  $n^{1/2} (   \hat{\theta}_n  - \theta_{0,n} )$ and a Gaussian distribution, of order $n^{-1/2}$ (up to a log factor).

There has been a recent interest for bounding the normal approximation of M-estimators, as we do here.
On connected topics, the normal approximation is quantified in \cite{pinelis2016optimal} for the Delta method, in \cite{anastasiou2020bounds}  for likelihood ratios and in \cite{anastasiou2019normal} for gradient descent. 
Considering now specifically M-estimators, a series of articles successfully addressed them: 
\cite{anastasiou2017bounds,anastasiou2018assessing,anastasiou2020multivariate,anastasiou2020wasserstein,anastasiouLey2017bounds,anastasiouReinert2017bounds,bentkus1997berry,pinelis2017optimal,shao2021berry}. 
These articles address not only the univariate case (for $\theta$) \cite{anastasiou2017bounds,
	anastasiouLey2017bounds,anastasiouReinert2017bounds,bentkus1997berry,pinelis2017optimal}, but also the general multivariate one \cite{anastasiou2018assessing,anastasiou2020multivariate,anastasiou2020wasserstein,shao2021berry}. In particular, some of these references exploit the characterization of the $L^1$ Wasserstein distance as a supremum of expectation differences, over Lipschitz functions. This enables to decompose the target Wasserstein distance into several terms that can be addressed independently with different approaches. This idea appears for instance in \cite[(9), (10) and (20)]{anastasiou2017bounds}, as well as some of the other articles above. We also rely on it, see \eqref{eq:basic:bound:Wasserstein} and \eqref{eq:using:lipschitzness}.

We shall now highlight the novelty of our results compared to the above articles. 
First, the references \cite{anastasiou2018assessing,anastasiou2020multivariate,anastasiouLey2017bounds,anastasiouReinert2017bounds,bentkus1997berry,pinelis2017optimal,shao2021berry} do not address the $L^1$ Wasserstein distance as we do.
Only \cite{anastasiou2017bounds,anastasiou2020wasserstein} do.
 In \cite{shao2021berry},  the distance is the supremum probability difference over convex sets, which is of the Berry-Esseen type. Earlier and similarly, \cite{bentkus1997berry,pinelis2017optimal}  considered the Kolmogorov distance in the univariate case. Also, \cite{anastasiouLey2017bounds,anastasiouReinert2017bounds} address Zolotarev-type distances based on supremums of expectation differences over absolutely continuous bounded test functions (and Lipschitz in \cite{anastasiouReinert2017bounds}, yielding the bounded-Wasserstein distance). 
 Similarly, \cite{anastasiou2018assessing,anastasiou2020multivariate} consider test functions that are bounded with bounded derivatives of various orders. 
 Remark that while the $L^1$ Wasserstein and Kolmogorov distances can be compared under regularity conditions and a priori moment bounds, using general comparison results typically worsens the quantitative estimates. 
 Note also that bounding the $L^1$ Wasserstein distance is stronger than in \cite{anastasiou2018assessing,anastasiou2020multivariate,anastasiouReinert2017bounds}, as it allows for a larger class of test functions.
Remark furthermore that Berry-Esseen-type and Kolmogorov distances may be less sensitive than Wasserstein distances to, for instance, the moments of $ \hat{\theta}_n - \theta_{0,n}$. Thus, the Wasserstein distances necessitate specific treatments compared to them (for instance, see the proof and use of Lemma \ref{lemma:proba:estimation:error:and:zero:gradient} here, or the terms in Theorem 2.1 in   \cite{anastasiouReinert2017bounds} involving the moments of $ \hat{\theta}_n - \theta_{0,n}$). 
 
In addition, we allow for general functions $M_n$, while most of the above references focus on maximum likelihood. Some arguments provided for maximum likelihood do carry over to general functions $M_n$, but it is not clear that this is the case for all of them. Also, most of the above references focus on independent observations (often also identically distributed) defining the function $M_n$ (with the exception of \cite{anastasiou2017bounds}), while we allow for $M_n$ stemming from dependent observations. Again, some but not all arguments for independent observations can be extended to dependent observations.  In the case of independent observations, as in \cite{anastasiou2020wasserstein} we shall rely on a result of Bonis \cite{bonis2020stein} to bound the rate of convergence in the multivariate central limit theorem. 

Furthermore, in comparison to \cite{anastasiou2017bounds,anastasiou2018assessing,anastasiou2020multivariate,anastasiou2020wasserstein,anastasiouLey2017bounds,anastasiouReinert2017bounds}, our general bound in Theorem \ref{theorem:general:bound} only depends on $M_n$ and its derivatives, and does not feature $ \hat{\theta}_n - \theta_{0,n}$. In contrast, most of the general bounds in these references contain moments of $ \hat{\theta}_n - \theta_{0,n}$ (see for instance Theorem 2.1 in \cite{anastasiouReinert2017bounds}). Hence, our general bound seems more convenient to apply to examples, particularly when  $ \hat{\theta}_n$ does not have an explicit expression, which is often the case. In agreement with this, in most of the examples provided by \cite{anastasiou2017bounds,anastasiou2018assessing,anastasiou2020multivariate,anastasiou2020wasserstein,anastasiouLey2017bounds,anastasiouReinert2017bounds}, $ \hat{\theta}_n$ has an explicit expression. As an exception, \cite{anastasiou2018assessing,anastasiouReinert2017bounds} address maximum likelihood estimation of the shape parameters of the Beta distribution. 
Finally, \cite{anastasiou2017bounds,anastasiou2018assessing,anastasiou2020multivariate,anastasiou2020wasserstein,anastasiouLey2017bounds,anastasiouReinert2017bounds} usually make the assumption that there is a unique $ \hat{\theta}_n$ satisfying \eqref{eq:hat:theta:general}, while Theorem \ref{theorem:general:bound} here holds for any $ \hat{\theta}_n$ satisfying \eqref{eq:hat:theta:general}. In many statistical models of interest, there is no guarantee that $M_n$ has a unique minimizer over $\Theta$, almost surely. 

The examples we address are representative of the flexibility of Theorem \ref{theorem:general:bound}. In particular we address general averages of independent functions in Section \ref{subsection:average:iid}. We treat logistic regression in Section \ref{subsection:logistic}, with a simple proof once Theorem \ref{theorem:general:bound} is established, which illustrates that this theorem is efficient even when $ \hat{\theta}_n$ does not have an explicit expression, and is not necessarily unique. Finally, in Section \ref{subsection:cross:validation} we address cross validation estimation of covariance parameters of Gaussian processes. This last example highlights our flexibility to dependent observations and to $M_n$ not stemming from a likelihood and even not being an average of functions of individual observations (most of the discussed references above consider these averages of functions for $M_n$). Again, $ \hat{\theta}_n$ has no explicit expression in this cross validation example.

The rest of the paper is organized as follows. Section \ref{section:general:bounds} provides the general technical conditions and the general bound of Theorem \ref{theorem:general:bound}, reducing the problem to  the asymptotic normality of 
the normalized gradient. It also discusses many references to address this  asymptotic normality in the probabilistic literature. Section \ref{section:applications} addresses the three examples discussed above. Some of the proofs are postponed to the appendix.

\section{General bounds} \label{section:general:bounds}

For a $\ell \times \ell$ matrix $A$, we write $\rho_{\ell}(A) \leq \dots \leq \rho_1(A) $ for its singular values, and for a symmetric matrix, we write $\lambda_{\ell}(A) \leq \dots \leq \lambda_1(A) $ for its eigenvalues. 

\subsection{Technical conditions} \label{subsection:technical:conditions}

For $u,v \in \mathbb{R}^p$, we write $[u,v] = \{ t u + (1-t) v ; t \in [0,1] \}$ and $(u,v) = \{ t u + (1-t) v ; t \in (0,1) \}$. We write $\Int{\Theta}$ for the interior of the parameter space $\Theta$. The next condition means that $\Theta$ is, so to speak, well-behaved. It can be checked that this condition holds for most common compact parameter spaces, in particular hypercubes, balls, ellipsoids and polyhedral sets.  

\begin{condition} \label{cond:Theta:covering}
There exist two constants $0 < C_{\Theta} < \infty$  and $0 < c'_{\Theta} < \infty$ such that for each $0 < \epsilon \leq c'_{\Theta}$, there exist $N \leq C_{\theta} \epsilon^{-p}$ and $\theta_1 , \ldots , \theta_N \in \Theta$ satisfying the following.  For each $\theta \in \Theta$, there exists $i \in \{1 , \ldots , N\}$ such that $( \theta , \theta_i ) \subseteq \Int{\Theta} $ and $||\theta - \theta_i|| \leq \epsilon$. 
\end{condition}

Then, the next condition basically consists in asking for enough integrability on the derivatives of $M_n$ to be able to commute expectation and derivation, which is usually established using the dominated convergence theorem. Remark that the conditions on the first two derivative orders will actually be implied by some of our later conditions, but we state them here independently for convenience of writing. 

\begin{condition}  \label{cond:smoothness}
	Consider $n \in \IN$. For $\theta \in \Theta$,  the random variable $M_n(\theta)$ is absolutely summable.
	Almost surely, the function $M_n$ is three times differentiable on $\Int{\Theta}$.  For $i,j,k \in \{1 , \ldots ,p\}$ and $\theta \in \ITheta$, the random variables $\partial  M_n (\theta) / \partial \theta_i $, $\partial^2 M_n (\theta) / \partial \theta_i \partial \theta_j $ and $\partial^3 M_n (\theta) / \partial \theta_i \partial \theta_j \partial \theta_k $ are absolutely summable. Furthermore, 
	\[
 \Esp \left(	\frac{\partial M_n (\theta) }{\partial \theta_i}  \right) 
 =
 \frac{\partial \Esp ( M_n (\theta) )}{\partial \theta_i}, \hspace{5mm}
	\Esp \left(	\frac{\partial^2 M_n (\theta) }{\partial \theta_i \partial \theta_j}  \right) 
	=
	\frac{\partial^2 \Esp ( M_n (\theta) )}{\partial \theta_i \partial \theta_j}
	\]
	and
		\[
	\Esp \left(	\frac{\partial^3 M_n (\theta) }{\partial \theta_i \partial \theta_j  \partial \theta_k}  \right) 
	=
	\frac{\partial^3 \Esp ( M_n (\theta) )}{\partial \theta_i \partial \theta_j \partial \theta_k}.
	\]
\end{condition}

The next condition means that, for a fixed $\theta$, $M_n(\theta)$ and $ \partial M_n(\theta) / \partial \theta_i$, $i \in \{1 , \ldots , p\}$, concentrate around their expectations at rate $n^{-1/2}$, with an exponential decay for deviations of order larger than $n^{-1/2}$.  Many tools from concentration inequalities (for instance \cite{boucheron2013concentration,chatterjee2014superconcentration}) enable to check this condition in specific settings (see for instance those of Section \ref{section:applications}). The rate $n$ in the exponential is sharp in general for averages of i.i.d. random variables. 

\begin{condition} \label{cond:concentration}
	There are constants $0 < c_{M} < \infty$, $0 < c'_{M} < \infty$ and $0 < C_{M} < \infty$ such that for $n \in \IN$ and $0 < \epsilon \leq c'_{M}$,
	\[
	\sup_{\theta \in \Theta} 
	\mathbb{P}
	( |M_n(\theta) - \Esp(M_n(\theta))  | \geq \epsilon )
	\leq C_M
	\exp( - n c_M \epsilon^2 )
	\]
and
\[	\sup_{\theta \in \ITheta} 
	\mathbb{P}
	( || \nabla M_n(\theta) - \Esp( \nabla M_n(\theta))  || \geq \epsilon )
	\leq C_M
	\exp( - n c_M \epsilon^2 ).
	\]
\end{condition}

For a function $f: \ITheta \to \IR$ and for $\theta \in \Int{\Theta}$, we write $\nabla f (\theta)$ the gradient column vector of $f$ at $\theta$ and we write $\nabla^2 f (\theta)$ the Hessian matrix of $f$ at $\theta$. The next condition is a control on the deviations of the derivatives of $M_n$ of order 1 and 2, that is uniform over $\ITheta$. Remark that the deviations that are controlled are of larger order than those in Condition \ref{cond:concentration}. Hence, again, the condition can be checked in many settings.

\begin{condition} \label{cond:large:dev}
		There are constants $0 < c_{d,1} < \infty$, $0 < C_{d,1} < \infty$ and $0 < C'_{d,1} < \infty$ such that for $n \in \IN$ and $K \geq C'_{d,1}$,
		\[
		\mathbb{P} \left(
		 \sup_{\theta \in \ITheta} 
		\left|  \left|
		\nabla M_n(\theta)
		\right| \right|
		\geq K
		\right)
		\leq 
		C_{d,1} n \exp(- c_{d,1} K  )
		\]
		and
	\[
	\mathbb{P} \left(
	\sup_{\theta \in \ITheta} 
	\max_{i,j =1}^p
	\left|  
	\frac{\partial^2 M_n(\theta)}{\partial \theta_i \partial \theta_j}
	\right| 
	\geq K
	\right)
	\leq 
	C_{d,1} n \exp(- c_{d,1} K  ).
	\]
\end{condition}

We then require the derivatives of order 1, 2 and 3 of $M_n$ to have bounded moments of order 1, 1 and 2.

\begin{condition} \label{cond:moment}
	There is a constant $C_{d,2}$ such that for $n \in \IN$,
	\begin{equation} \label{eq:cond:moments:un:deux}
	\sup_{\theta \in \ITheta}
	\mathbb{E} \left(
	\left| \left|
    \nabla M_n(\theta)
	\right| \right|
	\right) 
	\leq 
	C_{d,2},
	\hspace{3mm}
	\sup_{\theta \in \ITheta}
	\max_{i,j=1}^p
	\mathbb{E} \left(
	\left| 
	\frac{ \partial^2 M_n(\theta) }{ \partial \theta_i \partial \theta_j }
	\right| 
	\right) 
	\leq 
	C_{d,2}
	\end{equation}
and
	\begin{equation} \label{eq:cond:moments:trois}
\max_{j,k,\ell=1}^p
\mathbb{E} \left(
	\sup_{\theta \in \ITheta}
\left|
\frac{\partial^3 M_n(\theta)}{ \partial \theta_j \partial \theta_k \partial \theta_\ell }
\right|^2
\right) 
\leq 
C_{d,2}.
\end{equation}
\end{condition}

Above, the moments are for fixed $\theta$ for the order 1 and 2. The moments for the order 3 are uniform over $\ITheta$. Note that it can be seen from the proof of Theorem \ref{theorem:general:bound} that assuming uniformity only locally around $\theta_{0,n}$ (see Condition \ref{cond:global:identifiability}) would be sufficient. For instance, \cite{anastasiou2020multivariate} has a similar locally uniform moment bound on the third-derivatives of the log-likelihood function (see (R.C.3) there).

The next condition requires the variances of the derivatives of order 1 and 2 of $M_n$ to be of order $1/n$. This condition is natural and easy to check in many settings, for example for i.i.d. random variables.

\begin{condition} \label{cond:var}
		There is a constant $C_{\Var}$ such that for $n \in \IN$, $j,k \in \{ 1 , \ldots , p \} $, 
			\[
				\sup_{ \theta \in \ITheta }
			\max_{j=1}^p
		\Var \left(   \frac{ \partial M_n (\theta) }{  \partial \theta_j }  \right)
		\leq 
		\frac{C_{\Var}}{n}
		\]
		and
	\[
	\sup_{ \theta \in \Int{\Theta} }
\max_{j,k=1}^p
\Var \left(
\frac{\partial^2 M_n(\theta)}{ \partial \theta_j \partial \theta_k}
\right) 
\leq 
\frac{C_{\Var}}{n}.
\] 

\end{condition}

For $x \in \IR^p$ and $r \geq 0$, we let $B(x,r)$ be the closed Euclidean ball in $\IR^p$ with center $x$ and radius $r$. The next condition introduces the sequence of deterministic parameters $(\theta_{0,n})_{n \in \IN}$, to which $(\hat{\theta}_n)_{n \in \IN}$ is asymptotically close. In the applications of Sections \ref{subsection:logistic} and \ref{subsection:cross:validation}, $\theta_{0,n} = \theta_0$ does not depend on the sample size and determines the fixed unknown data generating process. Nevertheless, it is beneficial to allow for a $n$-dependent $\theta_{0,n}$, to cover general cases of misspecified models, for instance as in \cite{bachoc2020uniformly,berk2013valid,huber1967under,white1982maximum}. 

\begin{condition} \label{cond:global:identifiability}
There exists a sequence $(\theta_{0,n})_{n \in \IN }$ and a constant $0 < c_{\theta_0} < \infty$ such that for each $n \in \IN$, $B(\theta_{0,n} , c_{\theta_0} ) \subseteq \ITheta $.
We each $n \in \IN$, $\Esp( \nabla M_n (\theta_{0,n}) ) = 0$.
For each $r >0$ such that $\Theta \backslash B( \theta_{0,n} , r ) \neq \varnothing$, there exist constants $N_r \in \IN$ and $0 < c_r < \infty$ such that for $n \geq N_r$,
\[
\inf_{ \substack{\theta \in \Theta \\ ||\theta - \theta_{0,n}|| \geq  r } }
\left(
\Esp (M_{n}(\theta))
-
\Esp(
M_n(\theta_{0,n})
)
\right) 
\geq 
c_r.
\]
\end{condition}

Condition \ref{cond:global:identifiability} is a usual one in M-estimation: $\theta_{0,n}$ cancels out the expected gradient of $M_n$ and is asymptotically the minimizer of $\Esp(M_n) $, so to speak. 

\begin{remark}
	In Condition \ref{cond:var}, it is actually sufficient that the second inequality holds only for $\theta = \theta_{0,n}$. We state Condition \ref{cond:var} as it is only for convenience of writing, and because checking the inequality uniformly over $\theta$ in the bounded $\ITheta$ usually brings no additional difficulty.  
\end{remark}

Then, define the covariance matrix of the normalized gradient 
\begin{equation} \label{eq:cov:score:general}
	\bar{C}_{n,0} = \Cov( \sqrt{n} \nabla M_n(\theta_{0,n}) )
\end{equation}
and the expected Hessian
\begin{equation} \label{eq:mean:hessian:general}
	\bar{H}_{n,0} = \mathbb{E}(\nabla^2 M_n( \theta_{0,n})).
\end{equation}
The next condition requires the expected Hessian matrix of $M_n$ at $\theta_{0,n}$ to be asymptotically strictly positive definite. Similarly to Condition \ref{cond:global:identifiability}, this is a usual requirement for $\theta_{0,n}$ and $\hat{\theta}_n$ to be close at asymptotic rate $n^{-1/2}$. 

\begin{condition} \label{cond:local:identifiability}
There are constants $ 0 < c_{\theta_0,H} < \infty $ and $N_{\theta_0,H} \in \IN$ such that for $n \geq N_{\theta_0,H}$
\[
\lambda_p(  	\bar{H}_{n,0}  )  
\geq c_{\theta_0,H}.
\]
\end{condition}

We finally require the covariance matrix of the normalized gradient to be asymptotically strictly positive definite, so that the Gaussian limit in the central limit theorem is non-degenerate.  

\begin{condition} \label{cond:smallest:eigenvalue:covariance:nabla}
	There are constants $c_{\theta_0,\nabla} >0$ and $N_{\theta_0,\nabla} \in \IN$ such that for $n \geq N_{\theta_0,\nabla}$,
	\[
	 \lambda_p( \bar{C}_{n,0} ) \geq c_{\theta_0,\nabla}.
	\]
\end{condition}

\subsection{Reduction to the normal approximation of the normalized gradient}

We let $\mathcal{L}_1$ be the set of $1$-Lipschitz continuous functions from $\mathbb{R}^p$ to $\mathbb{R}$, that is the set of functions $g$ such that, for all $x_1 , x_2 \in \mathbb{R}^p$,
\[
|g(x_1) - g(x_2)| \leq || x_1 - x_2 ||.
\]

Then, for two random vectors $U$ and $V$ in $\mathbb{R}^p$, the $L^1$ Wasserstein distance between the distributions of $U$ and $V$ is
\[
\mathcal{W}_1(U,V) = 
\sup_{f \in \mathcal{L}_1}
| \mathbb{E}(f(U)) - \mathbb{E}(f(V)) |.
\] 
Equivalently, $\mathcal{W}_1(U,V)$ is also the well known $L^1$ optimal transport cost, according to the Kantorovitch-Rubinstein duality formula: 
\[
\mathcal{W}_1(U,V) = 
\inf_{ (\tilde{U},\tilde{V}) \sim \Pi(U,V) }
\mathbb{E}( || U - V|| ),
\]
where $\Pi(U,V)$ is the set of pairs of random vectors for which the first one is distributed as $U$ and the second one as $V$.

For a symmetric non-negative definite matrix $A$, we write $A^{1/2}$ for its unique symmetric non-negative definite square root. When $A$ is also invertible, we write $A^{-1/2} = (A^{1/2})^{-1} = (A^{-1})^{1/2}$. 

The next theorem is the main result of this paper. It can be checked, using standard arguments, that the conditions of Section \ref{subsection:technical:conditions} imply that $n^{1/2} (\hat{\theta}_n - \theta_{0,n})$ is asymptotically normally distributed, with asymptotic covariance matrix taking the ``sandwich'' form $ \bar{H}_{n,0}^{-1} \bar{C}_{n,0} \bar{H}_{n,0}^{-1}$. Equivalently, 	
$\bar{C}_{n,0}^{-1/2} \bar{H}_{n,0}
n^{1/2}  (\hat{\theta}_n - \theta_{0,n})$ converges to a standard Gaussian distribution. We are interested in the Wasserstein distance between the distribution of this latter random vector and the standard Gaussian one. We show that this distance is bounded by the sum of a term of order $n^{-1/2}$ (up to a log factor) and the distance between  $\bar{C}_{n,0}^{-1/2}
n^{1/2} \nabla M_n (\theta_{0,n})$ and the standard Gaussian distribution. The benefit on Theorem \ref{theorem:general:bound} is then that $\bar{C}_{n,0}^{-1/2}
n^{1/2} \nabla M_n (\theta_{0,n})$ is usually much easier to analyze than $\bar{C}_{n,0}^{-1/2} \bar{H}_{n,0}
n^{1/2}  (\hat{\theta}_n - \theta_{0,n})$, since it takes an explicit form and is not defined as a minimizer. In Section \ref{section:background:quantitative:CLT}, we discuss many existing possibilities to quantify the asymptotic normality of $\bar{C}_{n,0}^{-1/2}
n^{1/2} \nabla M_n (\theta_{0,n})$. 

\begin{theorem} \label{theorem:general:bound}
Assume that Conditions \ref{cond:Theta:covering} to \ref{cond:smallest:eigenvalue:covariance:nabla} hold. Consider $\hat{\theta}_n$ as in \eqref{eq:hat:theta:general}.
There are constants $0 < C_{\cW,1} < \infty$, $0 < C_{\cW,2} < \infty$ and $N_{\cW} \in \IN$ such that for $n \geq N_{\cW}$, with $Z$ following the standard Gaussian distribution on $\IR^p$,
\begin{align*}
	 \mathcal{W}_1(
	\bar{C}_{n,0}^{-1/2} \bar{H}_{n,0}
	\sqrt{n}  (\hat{\theta}_n - \theta_{0,n}) 
	, 
	Z  
	)
	\leq 
	\mathcal{W}_1
	\left( 
	 \bar{C}_{n,0}^{-1/2}
	\sqrt{n} \nabla M_n (\theta_{0,n}) , Z 
	\right)
	+
	C_{\cW,1} 
 \frac{(\log n)^{C_{\cW,2}}}{ \sqrt{n} } . 
\end{align*}
\end{theorem}

\begin{remark}
	In Theorem \ref{theorem:general:bound}, the bound on $\mathcal{W}_1(
	\bar{C}_{n,0}^{-1/2} \bar{H}_{n,0}
	n^{1/2}  (\hat{\theta}_n - \theta_{0,n}) 
	, 
	Z  
	)$
	directly provides a similar bound on 
	$\mathcal{W}_1(
	n^{1/2}  (\hat{\theta}_n - \theta_{0,n}) 
	, 
	Z _n 
	)$,
	where $Z_n$ follows the centered Gaussian distribution with covariance matrix $	\bar{H}_{n,0}^{-1} \bar{C}_{n,0} \bar{H}_{n,0}^{-1} $. Indeed the matrix $\bar{H}_{n,0}^{-1} \bar{C}_{n,0}^{1/2} $ is bounded and we can apply the well-known Lemma \ref{lemma:Lipschitz:Wasserstein} below. The same remark applies to Theorems \ref{thm_classical_m_bnd}, \ref{theorem:logistic:bound} and \ref{theorem:cross:validation:bound}, since the matrix $\bar{H}_{n,0}^{-1} \bar{C}_{n,0}^{1/2} $ is also bounded in these latter contexts (as is shown in the proofs). 
\end{remark}

\begin{lemma} \label{lemma:Lipschitz:Wasserstein}
	Let $U,V$ be two random vectors of $\IR^p$ and $h: \IR^p \to \IR^p$ be such that for $u,v \in \IR^p$, $||h(u) - h(v)|| \leq C ||u-v||$ with $0 < C < \infty$. Then $\cW_1( h(U) , h(V) ) \leq C \cW_1 (U,V)$.  
\end{lemma}

\subsection{Background on approximate normality for functions of many random variables} \label{section:background:quantitative:CLT}

Theorem \ref{theorem:general:bound} reduces the problem of proving a quantitative bound on distance to the Gaussian for a general M-estimator to proving the same statement for an explicit function of the data. We shall now describe some of the broad ideas for proving such statements, some of which will be used in the applications described in Section \ref{section:applications}. We do not aim at being exhaustive, and other techniques can also be used in this context. 

The abstract setting is to consider a random variable of the form $f(X_1,...,X_n)$ where the $X_i$ are random variables. The classical central limit theorem consists in taking the $X_i$ to be i.i.d. and $f$ to be a normalized sum. 

When $f$ is a sum, which arises for M-estimators of the form \eqref{eq:M:n:generic:average} (see Sections \ref{subsection:average:iid} and \ref{subsection:logistic}), there is a vast literature on quantitative central limit theorems, beyond the classical i.i.d. assumptions. For independent variables, we shall use here a very general result of Bonis \cite{bonis2020stein}, but many other results can be used in such a situation.

If $f$ is not a sum, but is approximately affine, and all variables have some influence on the value, we still expect approximate normality. This heuristic has been made rigorous by second-order Poincar\'e inequalities, which bound distances to the Gaussian when certain functions of the first and second derivatives are small. They have been introduced in the Gaussian setting by Chatterjee \cite{chatterjee2009poincare}, extended in \cite{nourdin2009poincare}, and analogues for general independent random variables via discrete second-order derivatives were studied in \cite{chatterjee2009stein,decreusefond2019dirichlet, duerinckx2021glauber}. Second-order Poincar\'e inequalities for non-Gaussian, non-independent random variables do not seem to have been yet addressed in the literature, and warrant further investigation. 

Another method for proving approximate normality in the Gaussian setting when the function $f$ is a multivariate polynomial is via the quantitative fourth moment theorem of Nourdin and Peccati \cite{nourdin2009stein}, which for example applies to U-statistics. When the polynomial is square-free and has low influences, it is possible to extend this phenomenon to more general i.i.d. random variables \cite{nourdin2010invariance}. The approach extends to non-independent functions of Gaussian variables, a result known as the quantitative Breuer-Major theorem \cite{nourdin2020breuer,nourdin2011breuer}. We refer to the monograph \cite{nourdin2012livre} for a thorough discussion of this approach. We shall use a variant of it in Section \ref{subsection:cross:validation}. 

For non-independent random variables, there have been successful implementations of variants of Stein's method, often in situations where there is some symmetry. Classical techniques include the exchangeable pairs method and the zero-bias transform, and we refer to \cite{ross11} for a survey.

\section{Applications} \label{section:applications}
\subsection{Minimization of averages of independent functions} \label{subsection:average:iid}

We now show how Theorem \ref{theorem:general:bound} applies to estimators provided by
$$ \label{eq:Mn:generic:average:in:Section}
M_n(\theta) = \frac{1}{n}\sum_{i=1}^n \rho(\theta, X_i),
$$
as in \eqref{eq:M:n:generic:average} with independent random vectors $X_1,\ldots,X_n$.  

We introduce the property of sub-Gaussianity, that holds for a large class of random variables, including Gaussian random variables, bounded random variables and uniformly log-concave random variables. 

\begin{definition}
A real-valued random variable $X$ is said to be sub-Gaussian with constant $\sigma^2$ if for any $t \geq 0$ we have
$$\E \left( \exp(t(X-\E[X])) \right) \leq \exp(t^2\sigma^2/2).$$
\end{definition}

The next theorem, based on Theorem \ref{theorem:general:bound}, provides a bound of order $n^{-1/2}$ (up to a log factor) in Wasserstein distance for the asymptotic normality of M-estimators based on \eqref{eq:M:n:generic:average}, under uniform sub-Gaussiannity for $\rho$ and its derivatives with respect to $\theta$. 

\begin{theorem} \label{thm_classical_m_bnd}
Assume that $X_1 , \ldots , X_n$ are independent.
Assume moreover that there  are constants $0 < \sigma^2 < \infty$ and $0 < E_{\sup} < \infty$ such that for any $i \in \{1, \ldots , n\}$, for any $j,k,\ell \in \{1 , \ldots ,p\}$,  for any $\theta_1 \in \Theta$, for any $\theta_2 \in \ITheta$, 
\begin{align} \label{eq:Y:subgaussian}
	& \text{for any} ~  Y \in
	\left\{ 
	\rho(\theta_1, X_i), 
	\partial \rho(\theta_2, X_i) / \partial \theta_j,
	\partial^2 \rho(\theta_2, X_i) / \partial \theta_j \partial \theta_k,
	\partial^3 \rho(\theta_2, X_i) / \partial \theta_j \partial \theta_k \partial \theta_\ell
	\right\}, \notag \\
	&
\text{$Y$ is sub-Gaussian with constant $\sigma^2$ and has absolute expectation bounded by $E_{\sup}$}. 	
\end{align}
Assume moreover that Conditions \ref{cond:Theta:covering}, \ref{cond:smoothness} and \ref{cond:global:identifiability} to \ref{cond:smallest:eigenvalue:covariance:nabla} hold. 
Consider $M_n$, $\hat{\theta}_n$, $\bar{C}_{n,0}$ and $\bar{H}_{n,0}$ as in \eqref{eq:M:n:generic:average}, \eqref{eq:hat:theta:general}, \eqref{eq:cov:score:general} and \eqref{eq:mean:hessian:general}.
Finally, assume that one of the two following conditions hold: either
\begin{enumerate}
\item There exist fixed constants $\lambda > 0$ and $C < \infty$ such that $$
\mathbb{E}
\left(
\exp \left( \lambda \sup_{ \theta \in \ITheta } 
||\nabla \rho(\theta, X_k)||
\right) 
\right)
 \leq C; \hspace{3mm} \mathbb{E}
 \left(
 \exp\left(\lambda \sup_{ \theta \in \ITheta } \left|\frac{\partial^2 \rho}{\partial \theta_i \partial \theta_j}(\theta, X_k)\right|\right)
 \right)
  \leq C$$
and
$$\mathbb{E}
\left(
\exp\left(\lambda \sup_{ \theta \in \ITheta } \left|\frac{\partial^3 \rho}{\partial \theta_i \partial \theta_j \partial \theta_\ell}(\theta, X_k)\right|\right)
\right)
 \leq C$$
for all $k \in \{1 , \ldots , n\}$ and $i, j , \ell \in \{1 , \ldots , p\}$.

Or

\item All the functions $||\nabla \rho(\cdot, x)||$, $\partial^2\rho(\cdot, x)/\partial \theta_i \partial \theta_j$ and $\partial^3\rho(\cdot, x)/\partial \theta_i \partial \theta_j\partial \theta_\ell$ have a modulus of continuity bounded by some function $\omega$, uniformly in $x \in \mathcal{X}$ and in $i,j, \ell \in \{1 , \ldots ,p \}$. 
\end{enumerate}
Then there are constants $0 < C_{\rho,1} < \infty$, $0 < C_{\rho,2} < \infty$ and $N_{\rho} \in \IN$ such that, for $n \geq N_{\rho}$, with $Z$ following the standard Gaussian distribution,
$$\mathcal{W}_1(
	\bar{C}_{n,0}^{-1/2} \bar{H}_{n,0}
	\sqrt{n}  (\hat{\theta}_n - \theta_{0,n}) 
	, 
	Z  
	)
	\leq \frac{C_{\rho,1} (\log n)^{C_{\rho,2}}}{\sqrt{n}}.$$
\end{theorem}

\begin{remark}
The sub-Gaussianity assumption \eqref{eq:Y:subgaussian} of Theorem \ref{thm_classical_m_bnd} on the partial derivatives of $\rho( \theta , X_i )$ with respect to $\theta$ can be checked based on the sub-Gaussianity of $X_1,\ldots,X_n$ only and on regularity properties of $\rho$.  

Indeed, it is known that if a random vector $V$ with values in $\mathbb{R}^k$ has components that are sub-Gaussian with constant $\sigma^2$, then for any $c$-Lipschitz function $f: \mathbb{R}^k \to \mathbb{R}$, the variable $f(V)$ is sub-Gaussian with constant at most of order $k c^2\sigma^2$. The dimensional prefactor can be eliminated for example when the components are independent and satisfy Talagrand's $L^2$ transport-entropy inequality \cite{gozlan}. 
Consider then the case where $X_1,\ldots,X_n$ are uniformly sub-Gaussian and for any $j,k,\ell \in \{1 , \ldots , p\}$, for any $f \in \left\{ \rho \hspace{1mm}, \partial \rho / \partial \theta_j, \hspace{1mm} \partial^2 \rho / \partial \theta_j \partial \theta_k, \hspace{1mm} \partial^3 \rho / \partial \theta_j \partial \theta_k \partial \theta_\ell \right\}$, $f$ is Lipschitz
in its second variable, uniformly in $\theta$, and $|f(\theta,x_{i,0})|$ is bounded, also uniformly in $\theta$, for some reference values $x_{i,0}$ of $X_i$, $i=1,\ldots,n$. 
In this case then the uniform sub-Gaussianity assumption \eqref{eq:Y:subgaussian} of Theorem \ref{thm_classical_m_bnd} holds. 

Note also that these latter assumptions are not minimal. For example, we could relax the Lipschitz assumption on the second derivatives into some quadratic growth. The assumptions on the third derivatives are much stronger than what is necessary to ensure \eqref{eq:cond:moments:trois} to streamline applications: one can check essentially the same conditions on all derivatives up to order three, rather than single out a weaker condition for third derivatives. 
\end{remark}

\begin{remark}
The two possible conditions 1 and 2 in Theorem \ref{thm_classical_m_bnd} are used to ensure that Condition 4 holds. There are other possible ways of verifying it, such as classical chaining techniques used to bound the suprema of stochastic processes when stochastic forms of continuity (in $\theta$) hold, see for example \cite[Chapter 8]{vershynin}. 
\end{remark}

\begin{proof}[Proof of Theorem \ref{thm_classical_m_bnd}]
First we must check that the conditions required by Theorem \ref{theorem:general:bound} are satisfied. By assumptions, this means checking conditions \ref{cond:concentration} to \ref{cond:var}. 

From the sub-Gaussianity and bounded expectation assumption \eqref{eq:Y:subgaussian}, we uniformly control moments of all order, and the first two parts of Condition \ref{cond:moment} hold. 
Condition \ref{cond:concentration} is an immediate consequence of the Gaussian concentration assumption and Chernoff's concentration bound. Condition \ref{cond:var} can be established using the fact that we wish to control the variances of averages of independent variables, and the uniform moment bounds. 

Finally, we need to check that Condition \ref{cond:large:dev} holds, assuming either 1 or 2 holds. If the first one holds, Condition \ref{cond:large:dev} is just a consequence of Markov's inequality. If the second one holds, by continuity, Condition \ref{cond:Theta:covering} and fixing some $\lambda >0$, and some $\epsilon > 0$ small enough, we have for any $k \in \{ 1 , \ldots , n \}$,
\begin{align*}
\mathbb{E}&
\left( 
\exp \left(\lambda \sup_{ \theta \in \ITheta } ||\nabla \rho(\theta, X_k)||
\right)
\right)
 \leq \mathbb{E}
 \left(
\exp \left(
\lambda \sup_{\theta_i, i \leq N} ||\nabla \rho(\theta_i, X_k)|| + \lambda \omega(\epsilon)
\right)
\right) \\
&\leq e^{\lambda \omega(\epsilon)} \sum_{i \leq N} \mathbb{E}(\exp(\lambda ||\nabla \rho(\theta_i, X_k)||)) \\
&\leq C',
\end{align*}
for some constant $0 < C' < \infty$, where the final bound uses the Gaussian concentration of $||\nabla \rho(\theta, X_k)||$ for fixed $\theta$ and the uniform bound on its expectation. 
The same reasoning applies for the second derivatives, and therefore Condition  \ref{cond:large:dev} holds with the same argument as when 1 holds. One can also check \eqref{eq:cond:moments:trois} with the same reasoning. 

Since Theorem \ref{theorem:general:bound} applies, we are reduced to understanding the asymptotic behavior of
$$\sqrt{n}\nabla M_n(\theta_{0,n}) = \frac{1}{\sqrt{n}}\sum_{i=1}^n \nabla \rho(\theta_{0,n}, X_i).$$
Hence we are in the setting of a quantitative central limit theorem for sums of independent random vectors. From the sub-Gaussianity assumption \eqref{eq:Y:subgaussian}, we see that the fourth moments of $\nabla \rho(\theta_{0,n}, X_i)$, $i=1,\ldots,n$, are uniformly bounded. Moreover, by Condition \ref{cond:smallest:eigenvalue:covariance:nabla}, this is not modified by multiplying these vectors by $\bar{C}_{n,0}^{-1/2} $. Hence we are considering a sum of independent random vectors with covariances summing to the identity matrix $I_p$, and we can apply the following statement to conclude the proof, which is a particular case of a result of Bonis \cite[Theorem 11]{bonis2020stein}.
\begin{proposition} \label{prop:CLT:logistic}
	Let $(Z_i)_{i=1,...,n}$ be a sequence of independent random vectors taking values in $\R^p$, each centered, and such that $\Cov(\sum_{i=1}^n Z_i) = n I_p$. Assume moreover that for any $i \in \{ 1 , \ldots ,n \}$, $\E[||Z_i||^4] \leq \beta$, for a given $0 < \beta < \infty$. Then 
	$$
	\mathcal{W}_1 \left( 
	 \frac{1}{\sqrt{n}} \sum_{i=1}^n Z_i, Z
	\right)
	 \leq \frac{ C_0(\beta^{3/2} + p\beta)}{ \sqrt{n}}
	$$
	where $Z$ is a standard Gaussian vector on $\IR^p$, and the constant $C_0$ is a numerical constant that does not depend on $p$ or on the distribution of the $Z_i$'s. 
\end{proposition}
\end{proof}

\subsection{Parameter estimation in logistic regression}
\label{subsection:logistic}

We shall now present the simple example of logistic regression, where Theorem \ref{thm_classical_m_bnd} is applied to a maximum likelihood estimator. 
We consider a deterministic sequence $(x_i)_{i \in \mathbb{N}}$ of vectors in $\mathbb{R}^p$. To match the assumptions of Theorem \ref{thm_classical_m_bnd}, we assume this sequence to be bounded.

\begin{condition} \label{cond:bounded:x:logistic}
There is a constant $0 < C_{x,1} < \infty$ such that for $i \in \IN$,
\[
 ||x_i|| \leq C_{x,1}.
\] 	
\end{condition}

As previously, we let $\Theta$ be a fixed compact subset of $\mathbb{R}^p$.  We let $\theta_0 \in \ITheta$ be fixed.
 We consider a sequence $(y_i)_{i \in \mathbb{N}}$ of independent random variables with, for $i \in \mathbb{N}$, $y_i \in \{ 0 , 1\}$ and 
\begin{equation} \label{eq:proba:theta:zero:logistic}
P( y_i = 1 ) = \frac{e^{x_i^\top \theta_0}}{ 1 + e^{x_i^\top \theta_0} }.
\end{equation}
We let, for $\theta \in \Theta$,
\[
p_{i,\theta} = \frac{e^{x_i^\top \theta}}{ 1 + e^{x_i^\top \theta} }.
\]

Hence, we are in the classical well-specified case where the parameter $\theta_0 \in \Theta$ characterizes the data generating process, or distribution, of $y_1 , \ldots , y_n$. 
The likelihood function of $y_i$ is, for $\theta \in \Theta$,
\[
\mathcal{L}(\theta , y_i )
=
p_{i,\theta}^{y_i} (1 - p_{i,\theta})^{1-y_i}.
\]

Minus the logarithm of the likelihood of $y_i$ is, for $\theta \in \Theta$,
\begin{align*}
	\rho( \theta,x_i,y_i )
	& =
	- y_i \log(p_{i,\theta}) - (1 - y_i) \log(1 - p_{i,\theta})
	\\
	& =
	- y_i x_i^\top \theta + \log \left( 1 + e^{x_i^\top \theta} \right).
\end{align*}

Hence minus the normalized log likelihood function is, for $\theta \in \Theta$,
\begin{equation} \label{eq:Mn:logistic}
M_n(\theta)
=
\frac{1}{n}
\sum_{i=1}^n
\left(
- y_i x_i^\top \theta + \log \left( 1 + e^{x_i^\top \theta} \right)
\right).
\end{equation}

Note that we do not have an explicit expression for the minimizer of $M_n$. 
 We have, for $\theta \in \ITheta$,
\begin{equation} \label{eq:gradient:logistic}
	\nabla M_n(\theta)  = 
	\frac{1}{n}
	\sum_{i=1}^n
	\left(
	- y_i x_i + 
	\frac{e^{x_i^\top \theta}}{1 + e^{x_i^\top \theta}} x_i
	\right)
	= 
	\frac{1}{n}
	\sum_{i=1}^n
	\left(
	- y_i x_i + 
	p_{i,\theta} x_i
	\right). 
\end{equation}
We also have, for $\theta \in \ITheta$,
\begin{equation} \label{eq:hessian:logistic}
	\nabla^2 M_n(\theta)  = 
	\frac{1}{n}
	\sum_{i=1}^n
	\frac{e^{x_i^\top \theta} (1 + e^{x_i^\top \theta})
		-e^{x_i^\top \theta} e^{x_i^\top \theta}
	}{(1 + e^{x_i^\top \theta})^2} x_i x_i^\top
 = 
	\frac{1}{n}
	\sum_{i=1}^n
	\frac{e^{x_i^\top \theta}}{(1 + e^{x_i^\top \theta})^2} x_i x_i^\top.
\end{equation}
Hence we see that $M_n(\theta)$ is convex with respect to $\theta$.
Next, we assume that the empirical second moment matrix of the $x_i$'s is asymptotically strictly positive definite. This type of condition is common with logistic regression \cite{bachoc2020uniformly,Fahrmeir90,lv14model} and enables to have asymptotic identifiability (Condition \ref{cond:local:identifiability}). 

\begin{condition} \label{cond:lambda:inf:cov:x:logistic}
There are constants $0 < c_{x,2} < \infty$ and $N_{x,2} \in \IN$ such that, for $n \geq N_{x,2}$,
\[
\lambda_p
\left(
\frac{1}{n}
\sum_{i=1}^n x_i x_i^\top 
\right)
\geq c_{x,2}.
\]
\end{condition}

We can now state the Wasserstein bound on the asymptotic normality of the maximum likelihood estimator, in logistic regression. To our knowledge, this is the first established rate of convergence of asymptotic normality in logistic regression. 

\begin{theorem} \label{theorem:logistic:bound}
	Assume that $\Theta$ satisfies Condition \ref{cond:Theta:covering}.
	Assume that Conditions \ref{cond:bounded:x:logistic} and \ref{cond:lambda:inf:cov:x:logistic} hold. 
	Consider $M_n$ in \eqref{eq:Mn:logistic}, $\hat{\theta}_n$ as in \eqref{eq:hat:theta:general}, $\theta_0$ as defined in \eqref{eq:proba:theta:zero:logistic},
	$\bar{C}_{n,0}$ as in \eqref{eq:cov:score:general} and $\bar{H}_{n,0}$ as in \eqref{eq:mean:hessian:general}.
	Then, there are constants $0 < C_{\text{log},1} < \infty$, $0 < C_{\text{log},2} < \infty$ and $N_{\text{log}} \in \IN$ such that for $n \geq N_{\text{log}}$, with $Z$ following the standard Gaussian distribution on $\IR^p$,
	\begin{align*}
		\mathcal{W}_1 \left(
		\bar{C}_{n,0}^{-1/2} \bar{H}_{n,0}
		\sqrt{n}  (\hat{\theta}_n - \theta_{0}) 
		, 
		Z  
		\right)
		\leq 
		C_{\text{log},1} 
		\frac{(\log n)^{C_{\text{log},2}}}{ \sqrt{n} } . 
	\end{align*}
\end{theorem}

\subsection{Covariance parameter estimation for Gaussian processes by cross validation} \label{subsection:cross:validation}

Our last example stems from the field of spatial statistics \cite{Bachoc2013cross,bachoc14asymptotic,bachoc2020spatial,chiles2009geostatistics,cressie2015statistics,hallin2009local,wackernagel2013multivariate,zhang04inconsistent,Zhang2010kriging}. The goal is to illustrate the benefit of Theorem \ref{theorem:general:bound} to a situation where the observations are dependent and where $M_n$ does not correspond to a likelihood. We stress that $\hat{\theta}_n$ has no explicit expression. 

We consider a sequence $(x_i)_{i \in \mathbb{N}}$ of deterministic vectors in $\mathbb{R}^d$, that we call observation points. Then, for $n \in \IN$, the observed data consist in a vector $y^{(n)}$ of size $n \times 1$ which component $i$ is $\xi( x_i )$, where $\xi : \mathbb{R}^d \to \mathbb{R}$ is a centered Gaussian process. 

We are interested in the parametric estimation of the correlation function of $\xi$, based on a parametric set of stationary correlation functions $\{  k_{\theta} ; \theta \in \Theta \}$, where for $\theta \in \Theta$, $k_{\theta} : \mathbb{R}^d \to \mathbb{R}$ and $(u , v) \in \mathbb{R}^{2d} \mapsto k_{\theta}(u-v)$ is a correlation function. For an introduction to usual parametric sets of stationary correlation functions in spatial statistics, we refer for instance to  \cite{Bachoc2021asymptotic,chiles2009geostatistics,cressie2015statistics,genton2015cross,wackernagel2013multivariate}.

As an estimator for $\theta$, we consider the minimization of the average of square leave-one-out errors, letting, for $\theta \in \Theta$,
\begin{equation*} 
	M_{n}(\theta)
	=
	\frac{1}{n}
	\sum_{i=1}^n
	\left(
	y^{(n)}_i
	-
	\mathbb{E}_{\theta}
	( y^{(n)}_i | y^{(n)}_{-i} )
	\right)^2.
\end{equation*}
Above, $y^{(n)}_{-i}$ is obtained from $y^{(n)}$ by deleting the component $i$ and $\mathbb{E}_{\theta}( \cdot | \cdot )$ means that the conditional expectation is computed as if the Gaussian process $\xi$ had correlation function $(u , v) \in \mathbb{R}^{2d} \mapsto k_{\theta}(u-v)$.
Now, for $\theta \in \Theta$, let $R_{n,\theta}$ be the $n \times n$ matrix with coefficient $i,j$ equal to $k_{\theta}(x_i - x_j)$, that is, the correlation matrix of $y^{(n)}$ under correlation function given by $k_{\theta}$.  Then, from for instance \cite{Bachoc2013cross,dubrule83cross,Zhang2010kriging} (to which we refer for more background and discussions on cross validation for Gaussian processes), we have
\begin{equation}  \label{eq:Mntheta:CV}
	M_{n}(\theta)
	=
	\frac{1}{n}
	y^{(n)\top}
	R_{n,\theta}^{-1}
	\mathrm{diag}(R_{n,\theta}^{-1})^{-2}
	R_{n,\theta}^{-1}
	y^{(n)},
\end{equation}
where $\mathrm{diag}(M)$ is obtained by setting the off-diagonal elements of a square matrix $M$ to zero. 

For $n \in \IN$, we let $\theta_{0,n} = \theta_0$, where $\theta_0$ is a fixed element of $\ITheta$ such that $\xi$ has correlation function $k_{\theta_0}$, which also implies that $y^{(n)}$ has correlation matrix $R_{n,\theta_0}$. This corresponds to a well-specified parametric set of correlation functions. 
The next condition means that we consider the increasing-domain asymptotic framework, where the sequence of observation points is unbounded, with a minimal distance between any two distinct points \cite{bachoc14asymptotic,cressie2015statistics,mardia84maximum}.  

\begin{condition} \label{cond:minimal:distance}
There is a constant $c_{x} >0$ such that for $i,j \in \IN$, $i \neq j$,
\begin{equation*} 
	|| x_i - x_j ||  
	\geq c_x.
\end{equation*}
\end{condition}

The next condition is a lower bound on the smallest eigenvalues of the correlation matrices from the parametric model. Given the increasing-domain asymptotic framework (Condition \ref{cond:minimal:distance}), this lower bound indeed holds for a large class of families of stationary correlation functions \cite{bachoc14asymptotic,bachoc2016smallest}.

\begin{condition} \label{cond:smallest:eigen:value}
	There is a constant $0 < c_{R,1} < \infty$ such that
	\[
	\inf_{n \in \mathbb{N}} \inf_{\theta \in \Theta} 
	\lambda_n( R_{n,\theta} )
	\geq c_{R,1}.
	\]
\end{condition}

Next, we assume a third-order smoothness with respect to $\theta$ as well as a decay of the correlation at large distance. As before, many families of stationary correlation functions do satisfy this. 

\begin{condition} \label{cond:cross:validation:smoothness:and:bounds}
	For any $x \in \IR^d$, $k_{\theta}(x)$ is three times continuously differentiable with respect to $\theta$ on $\ITheta$. There exist constants $0 < C_{R,2} < \infty$ and $0 < c_{R,2} < \infty$ such that for $\theta \in \Theta$, for $x \in \IR^d$,	
	\begin{equation} \label{eq:bound:R}
		\left|
		k_{\theta}(x)
		\right|
		\leq 
		\frac{C_{R,2}}{1+||x||^{d+c_{R,2}}},
		~ ~
		n \in \mathbb{N}
	\end{equation}
	and for $\theta \in \ITheta$, for $x \in \IR^d$,	
	\begin{equation} \label{eq:bound:R:derivatives}
		\max_{\substack{k \in \{1,2,3 \} \\ i_1,\ldots,i_k \in \{1 , \ldots , p \}}}
		\left|
		\frac{\partial^k}{\partial \theta_{i_1},\ldots,\partial \theta_{i_k}}
		k_{\theta}(x)
		\right|
		\leq 
		\frac{C_{R,2}}{1+||x||^{d+c_{R,2}}},
		~ ~
		n \in \mathbb{N}.
	\end{equation}
\end{condition}

The next condition is interpreted as a global identifiability of the correlation parameter. This condition is already made in the increasing-domain asymptotic literature on cross validation and is not restrictive on the sequence $(x_i)_{i \in \IN}$ and the set $\{ k_{\theta} \}$ \cite{bachoc14asymptotic,bachoc2020asymptotic}.

\begin{condition} \label{cond:cross:validation:global:identifiability}
	For all $\mathcal{X} >0$, there are constants $0 < c_{\mathcal{X}} < \infty$ and $N_{\mathcal{X}} \in \IN$ such that for $n \geq N_{\mathcal{X}}$,
	\begin{equation*}
		\inf_{\substack{\theta \in \Theta \\ || \theta - \theta_0 || \geq \mathcal{X} }}
		\frac{1}{n}
		\sum_{i,j=1}^n
		\left(
	k_{\theta}(x_i - x_j)
		-
k_{\theta_0}(x_i - x_j)
		\right)^2
		\geq c_{\mathcal{X}}.
	\end{equation*}
\end{condition}

Finally, the last condition is interpreted as a local identifiability of the correlation parameter around $\theta_0$. Its discussion is similar to the previous one. 

\begin{condition} \label{cond:cross:validation:local:identifiability}
	For all $\alpha_1,\ldots,\alpha_p \in \mathbb{R}$, with $\alpha_1^2 + \dots + \alpha_p^2 >0$, there are constants $0 < c_{\alpha} < \infty$ and $N_{\alpha} \in \IN$ such that for $n \geq N_{\alpha}$,
	\begin{equation*}
		\frac{1}{n}
		\sum_{i,j=1}^n
		\left(
		\sum_{\ell=1}^p
		\alpha_\ell
		\frac{ \partial  k_{\theta_0}  (x_i - x_j) }{ \partial \theta_\ell}
		\right)^2
		\geq c_{\alpha}.
	\end{equation*}
\end{condition}

Under the above conditions, it is known from \cite{bachoc14asymptotic,bachoc2020asymptotic} that $n^{1/2} ( \hat{\theta}_n  - \theta_0 )$ converges in distribution to a centered Gaussian vector with covariance matrix $\bar{H}_{n,0}^{-1} 	\bar{C}_{n,0}  \bar{H}_{n,0}^{-1} $, with the notation of \eqref{eq:cov:score:general} and \eqref{eq:mean:hessian:general}. Based on Theorem \ref{theorem:general:bound}, we can show that the rate of this convergence is $n^{-1/2}$ (up to a log factor) in Wasserstein distance. To the best of our knowledge, this is the first result of this kind for
cross validation estimation for spatial Gaussian processes. 
We remark that Theorem \ref{theorem:general:bound} also enables to address maximum likelihood
estimation of covariance parameters (see for instance \cite{bachoc14asymptotic,cressie2015statistics}), but we focus on cross validation for the sake of brevity and to highlight the benefits of  Theorem \ref{theorem:general:bound} beyond maximum likelihood.

\begin{theorem} \label{theorem:cross:validation:bound}
	Assume that $\Theta$ satisfies Condition \ref{cond:Theta:covering}.
	Assume that Conditions \ref{cond:minimal:distance} to \ref{cond:cross:validation:local:identifiability} hold. 
	Consider $M_n$ in \eqref{eq:Mntheta:CV}. 
	Consider then $\hat{\theta}_n$ as in \eqref{eq:hat:theta:general}, $\theta_0$ as defined after \eqref{eq:Mntheta:CV},
$\bar{C}_{n,0}$ as in \eqref{eq:cov:score:general} and $\bar{H}_{n,0}$ as in \eqref{eq:mean:hessian:general}.
	Then, there are constants $0 < C_{\text{CV},1} < \infty$, $0 < C_{\text{CV},2} < \infty$ and $N_{\text{CV}} \in \IN$ such that for $n \geq N_{\text{CV}}$, with $Z$ following the standard Gaussian distribution on $\IR^p$,
	\begin{align*}
		\mathcal{W}_1 \left(
		\bar{C}_{n,0}^{-1/2} \bar{H}_{n,0}
		\sqrt{n}  (\hat{\theta}_n - \theta_{0}) 
		, 
		Z  
		\right)
		\leq 
		C_{\text{CV},1} 
		\frac{(\log n)^{C_{\text{CV},2}}}{ \sqrt{n} } . 
	\end{align*}
\end{theorem}

\appendix

\section{Proofs for Section \ref{section:general:bounds}}

\begin{lemma} \label{lemma:uniform:concentration:Mn}
Assume that Conditions \ref{cond:Theta:covering} to \ref{cond:moment} hold.
Then there are constants $0 < c_{M,1} < \infty$, $0 < c'_{M,1} < \infty$, $0 < C_{M,1} < \infty$ and $0 < C'_{M,1} < \infty$ such that, for $0 < t \leq c'_{M,1}$ and $K \geq C'_{M,1}$, 
	\begin{equation*}
		\mathbb{P} \left( \sup_{\theta \in \Theta} 
		\left|
		M_n(\theta) - \mathbb{E}( M_n( \theta ) )
		\right|
		\geq t
		\right)
		\leq 
		C_{M,1} \frac{K^p}{t^p}
		\exp( - n c_{M,1} t^2 )
		+
		C_{M,1} n \exp(- c_{M,1} K  ).
	\end{equation*}
\end{lemma}

\begin{proof}[Proof of Lemma \ref{lemma:uniform:concentration:Mn}]
	From Condition \ref{cond:Theta:covering}, and with $c'_M$ and $C'_{d,1}$ from  Conditions \ref{cond:concentration} and \ref{cond:large:dev}, there exists a constant $C_{\Theta,2}$ such that for $ 0 < r \leq c'_M / 2 C'_{d,1}$, there exists $N \leq C_{\Theta,2} r^{-p} $ and $ S_{r} = \{ \theta_1 , \ldots , \theta_N \} \subseteq \Theta$ such that for each $\theta \in \Theta$, there exists $i \in \{1 , \ldots , N\}$ such that $( \theta , \theta_i ) \subseteq \Int{\Theta} $ and $||\theta - \theta_i|| \leq r$. We then have, for each $K \geq C'_{d,1}$, $ 0 < t \leq c'_M$, using the mean value theorem,
	\begin{align} 
		\mathbb{P} \left( \sup_{\theta \in \Theta} 
		\left|
		M_n(\theta) - \mathbb{E}( M_n( \theta ) )
		\right|
		\geq t
		\right)
		& \leq 
		\mathbb{P} \left( \max_{\theta \in S_{t/2K}} 
		\left|
		M_n(\theta) - \mathbb{E}( M_n( \theta ) )
		\right|
		\geq \frac{t}{2}
		\right)
		\notag
		\\
		& +
		\mathbb{P} \left( \sup_{\theta \in \ITheta} 
		\left|  \left|
		\nabla M_n(\theta)
		\right| \right|
		\geq \frac{K}{2}
		\right)
		+
		\mathbb{P} \left( \sup_{\theta \in \ITheta} 
		\left|  \left|
		\nabla \mathbb{E}( M_n(\theta))
		\right| \right|
		\geq \frac{K}{2}
		\right).
		\notag
		\end{align}
	Hence, because $\nabla \mathbb{E}( M_n(\theta)) = \mathbb{E}( \nabla  M_n(\theta))$ is bounded from Conditions \ref{cond:smoothness} and \ref{cond:moment}, and using a union bound,  there is a constant $C'_{d,1} \leq C_1 < \infty$	such that when $K \geq C_1$, $0 < t \leq c'_M$, we obtain 
	\begin{align} \label{eq:for:bounding:deviation:En}
		\mathbb{P} \left( \sup_{\theta \in \Theta} 
	\left|
	M_n(\theta) - \mathbb{E}( M_n( \theta ) )
	\right|
	\geq t
	\right)	& \leq 
		\frac{C_{\Theta,2} 2^p K^p}{t^p}
		\max_{\theta \in \Theta}
		\mathbb{P} \left( 
		\left|
		M_n(\theta) - \mathbb{E}( M_n( \theta ) )
		\right|
		\geq \frac{t}{2}
		\right)
		\\
		& +
		\mathbb{P} \left( \sup_{\theta \in \Theta} 
		\left|  \left|
		\nabla M_n(\theta)
		\right| \right|
		\geq \frac{K}{2}
		\right).  \notag
	\end{align}
	
Hence, using Conditions \ref{cond:concentration} and \ref{cond:large:dev},
we obtain, for  $ 0 < t \leq c'_M$ and $K \geq C_1$, 
	\begin{equation*} 
		\mathbb{P} \left( \sup_{\theta \in \Theta} 
		\left|
		M_n(\theta) - \mathbb{E}( M_n( \theta ) )
		\right|
		\geq t
		\right)
		\leq 
		C_{\Theta,2} 2^p \frac{K^p}{t^p}
		C_M
		\exp( - n c_M t^2/4 )
		+
		C_{d,1} n \exp(- c_{d,1} K / 2  ).
	\end{equation*}
This concludes the proof.
\end{proof}

\begin{lemma} \label{lemma:uniform:concentration:nabla}
	Assume that Conditions \ref{cond:Theta:covering} to \ref{cond:moment} hold.
	Then there are constants $0 < c_{\nabla,1} < \infty$, $0 < c'_{\nabla,1} < \infty$, $0 < C_{\nabla,1} < \infty$ and $0 < C'_{\nabla,1} < \infty$ such that, for $0 < t \leq c'_{\nabla,1}$ and $K \geq C'_{\nabla,1}$, 
	\begin{equation*} 
		\mathbb{P} \left( \sup_{\theta \in \Theta} 
		\left| \left|
		\nabla M_n(\theta) - \mathbb{E}( \nabla M_n( \theta ) )
		\right|
		\right|
		\geq t
		\right)
		\leq 
		C_{\nabla,1} \frac{K^p}{t^p}
		\exp( - n c_{\nabla,1} t^2 )
		+
		C_{\nabla,1} n \exp(- c_{\nabla,1} K  ).
	\end{equation*}
\end{lemma}

\begin{proof}[Proof of Lemma \ref{lemma:uniform:concentration:nabla}]
The proof is identical to that of Lemma \ref{lemma:uniform:concentration:Mn}.
\end{proof}

\begin{lemma} \label{lemma:proba:estimation:error}
	Assume that Conditions \ref{cond:Theta:covering} to \ref{cond:moment} and \ref{cond:global:identifiability} hold.
	For any $r >0$, there are constants $0 < c_{\hat{\theta},r} < \infty$ and $0 < C_{\hat{\theta},r} < \infty$ such that
	\[
	\mathbb{P}(  ||  \hat{\theta}_n - \theta_{0,n}  || \geq r )	
	\leq 
	C_{\hat{\theta},r}
	n
	\exp(- c_{\hat{\theta},r}  n^{1/4}  ).
	\]
\end{lemma}

\begin{proof}[Proof of Lemma \ref{lemma:proba:estimation:error}]
	The event $||  \hat{\theta}_n - \theta_{0,n}  || \geq r$ implies 
	\[
	\inf_{ \substack{\theta \in \Theta \\ ||\theta - \theta_{0,n}|| \geq  r } }
	\left(
	M_{n}(\theta)
	-
	M_n(\theta_{0,n})
	\right) 
	\leq 
	0.
	\]
From Condition \ref{cond:global:identifiability} and the triangle inequality, this implies, with a constant $0 < c_1 < \infty$, for $n$ large enough,
\[
\sup_{\theta \in \Theta}
\left|
M_n( \theta ) - \Esp( M_n(\theta) )
\right|
\geq 
c_1. 
\]

Hence 
\[
	\mathbb{P}(  ||  \hat{\theta}_n - \theta_{0,n}  || \geq r )	
\leq 
\mathbb{P} 
\left(
\sup_{\theta \in \Theta}
\left|
M_n( \theta ) - \Esp( M_n(\theta) )
\right|
\geq 
c_1
\right).
\]
Using now Lemma \ref{lemma:uniform:concentration:Mn} with $K = n^{1/4}$ and $n$ large enough, we obtain, for some constants $0< c_2 < \infty$, $0 < C_2 < \infty$, $0< c_3 < \infty$ and $0 < C_3 < \infty$, for $n$ large enough,
\begin{equation*}
	\mathbb{P}(  ||  \hat{\theta}_n - \theta_{0,n}  || \geq r )	
	 \leq 
	C_{2} n^{p/4}
	\exp( - n c_2  )
	+
	C_2 n \exp(- c_2 n^{1/4}  )
 \leq 
	C_3
	n
	\exp(- c_3 n^{1/4}  ). 
\end{equation*}
\end{proof}

\begin{lemma} \label{lemma:Hessian:lower:bound}
	Assume that Conditions  \ref{cond:smoothness}, \ref{cond:moment} and \ref{cond:local:identifiability} hold.
There exist constants $0 < c_{\nabla^2,1} < \infty$, $0 < c'_{\nabla^2,1} < \infty$ and $N_{\nabla^2,1} \in \IN$ such that for $n \geq N_{\nabla^2,1} $ 
\[
\inf_{ \substack{\theta \in \ITheta \\ ||  \theta - \theta_{0,n} || \leq c'_{\nabla^2,1} } }
\lambda_p(  \mathbb{E}( \nabla^2 M_n(\theta) ) )
\geq c_{\nabla^2,1}.
\] 	
\end{lemma}

\begin{proof}[Proof of Lemma \ref{lemma:Hessian:lower:bound}]
Condition \ref{cond:moment}, together with the fact that we can exchange derivatives and expectation for $M_n$ (Condition \ref{cond:smoothness}) imply that the derivatives of $\Esp( \nabla^2 M_n )$ are bounded uniformly in $\theta \in \ITheta$. Hence, from Condition \ref{cond:local:identifiability}, we can conclude the proof. 
\end{proof}

\begin{lemma} \label{lemma:nabla:Mn:larger:than}
		Assume that Conditions  \ref{cond:smoothness}, \ref{cond:moment},  \ref{cond:global:identifiability} and \ref{cond:local:identifiability} hold.
	There are constants $0 < c_{\nabla,2} < \infty$, $0 < c'_{\nabla,2} <   \infty$ and $N_{\nabla,2} \in \IN$ such that for $n \geq N_{\nabla,2}$, for $|| \theta - \theta_{0,n} || \leq c'_{\nabla,2}$,
\[
	||
	\mathbb{E}(\nabla M_n(\theta))
	||
	\geq 
	c_{\nabla,2} || \theta - \theta_{0,n} ||.
\]
\end{lemma}

\begin{proof}[Proof of Lemma \ref{lemma:nabla:Mn:larger:than}]
Using Lemma \ref{lemma:Hessian:lower:bound} and $\mathbb{E}( \nabla^2 M_n(\theta) ) = \nabla^2 \mathbb{E}(  M_n(\theta) ) $ (Condition \ref{cond:smoothness}), we have, for $|| \theta - \theta_{0,n} || \leq c'_{\nabla^2,1}$ and for $n$ large enough,
\[
||
\nabla
\mathbb{E}(M_n(\theta))
-
\nabla
\mathbb{E}( M_n(\theta_{0,n}))
||
~
||\theta - \theta_{0,n}||
\geq 
\left(
\nabla
\mathbb{E}( M_n(\theta))
-
\nabla
\mathbb{E}( M_n(\theta_{0,n}))
\right)^\top
(\theta - \theta_{0,n})
\geq
c_{\nabla^2,1} || \theta - \theta_{0,n} ||^2.
\]
From Conditions \ref{cond:smoothness} and \ref{cond:global:identifiability},
\[
\nabla \mathbb{E}( M_n(\theta_{0,n})) = 0.
\] 
Hence we have,  for $|| \theta - \theta_{0,n} || \leq c'_{\nabla^2,1}$ and for $n$ large enough,
\[
||
\nabla
\mathbb{E}( M_n(\theta))
||
\geq 
c_{\nabla^2,1} || \theta - \theta_{0,n} ||.
\]
We conclude from Condition \ref{cond:smoothness}.
\end{proof}

\begin{lemma} \label{lemma:proba:estimation:error:and:zero:gradient} 
	Assume that Conditions \ref{cond:Theta:covering} to \ref{cond:moment}, \ref{cond:global:identifiability} and \ref{cond:local:identifiability} hold. Recall $c_{\theta_0}$ from Condition \ref{cond:global:identifiability}.
For any constant $\gamma_1>0$, there are constants $0 < c_{\nabla,\hat{\theta},1} < \infty$, $0 < c'_{\nabla,\hat{\theta},1} \leq c_{\theta_0}$, $0 < C_{\nabla,\hat{\theta},1} < \infty$ and $N_{\nabla,\hat{\theta},1} \in \IN$ such that for $n \geq N_{\nabla,\hat{\theta},1}$ and $t \leq c'_{\nabla,\hat{\theta},1}$, 
\begin{align*}
	\mathbb{P}
	\left(
	\nabla M_n( \hat{\theta}_n ) = 0
	,
t \leq 	||  \hat{\theta}_n - \theta_{0,n} ||
\leq c_{\theta_0}
	\right)
	\leq &
	C_{\nabla,\hat{\theta},1} 
	\frac{\log(n)^{ p \gamma_1 }}{t^p} 
	\exp( - n c_{\nabla,\hat{\theta},1} t^2  )
	\\
	& +
	C_{\nabla,\hat{\theta},1}  n
	\exp( - c_{\nabla,\hat{\theta},1} (\log n)^{\gamma_1}  )
	+ C_{\nabla,\hat{\theta},1}   n \exp( - c_{\nabla,\hat{\theta},1} n^{1/4} ).
\end{align*}
\end{lemma}

\begin{proof}[Proof of Lemma \ref{lemma:proba:estimation:error:and:zero:gradient}]
		Recall $c'_{\nabla,2}$ from Lemma \ref{lemma:nabla:Mn:larger:than}.
		For $0 < t < c'_{\nabla,2}$, we have, using Lemmas \ref{lemma:proba:estimation:error} and \ref{lemma:nabla:Mn:larger:than},
		\begin{align} \label{eq:for:bounding:proba:deviation:hat:theta}
			&\mathbb{P}
			\left(
			\nabla M_n( \hat{\theta}_n ) = 0
			,
		t \leq 	||  \hat{\theta}_n - \theta_{0,n} ||
		\leq c_{\theta_0}
			\right)
			\leq 
			\mathbb{P}
			\left(
			\inf_{ \theta \in B(\theta_{0,n} , c'_{\nabla,2}) \backslash B(\theta_{0,n} , t) }
			|| \nabla M_n(\theta ) || = 0
			\right)
			+ \mathbb{P}
			\left(
			|| \hat{\theta}_n - \theta_{0,n} || \geq c'_{\nabla,2}
			\right)
			\notag
			\\
			& \hspace{1cm} \leq 
			\mathbb{P}
			\left(
			\sup_{ \theta \in \ITheta }
			|| \nabla M_n(\theta ) - \mathbb{E} ( \nabla M_n(\theta ) ) || 
			\geq 
			c_{\nabla,2} t
			\right) +
			C_{\hat{\theta},c'_{\nabla,2}}
		n
		\exp(- c_{\hat{\theta},c'_{\nabla,2}}  n^{1/4}  ). 
		\end{align}
		For any constant $0 < \gamma_1 < \infty$, we can now use Lemma \ref{lemma:uniform:concentration:nabla} with $K = (\log n)^{\gamma_1}$ to obtain, 
		for $0 < t < \min( c'_{\nabla,2} , c'_{\nabla,1} )$,
		for $n$ large enough,
		\begin{align} \label{eq:bound:proba:estimation:error:aut}
			&\mathbb{P}
			\left(
			\nabla M_n( \hat{\theta}_n ) = 0
			,
			t \leq 	||  \hat{\theta}_n - \theta_{0,n} ||
		\leq c_{\theta_0}
			\right)
			 \leq 
			C_{\nabla,1}
			\frac{\log(n)^{ p \gamma_1 }}{c_{\nabla,2}^p t^p} 
			\exp( - n c_{\nabla,1} c_{\nabla,2}^2 t^2  )
			\\
			& \hspace{1cm} +
			C_{\nabla,1} n
			\exp( - 	c_{\nabla,1} (\log n)^{\gamma_1}  )
			+ 	C_{\hat{\theta},c'_{\nabla,2}}
			n
			\exp(- c_{\hat{\theta},c'_{\nabla,2}}  n^{1/4}  ).
		\end{align}
This concludes the proof.
\end{proof}

\begin{proof}[Proof of Theorem \ref{theorem:general:bound}]
	
	From the triangle inequality, we have
	\begin{align} \label{eq:basic:bound:Wasserstein}
		\mathcal{W}_1
		\left(
		\bar{C}_{n,0}^{-1/2} \bar{H}_{n,0}  \sqrt{n} (\hat{\theta}_n - \theta_{0,n}) 
		,
		Z 
		\right)
		\leq &
		\mathcal{W}_1
		\left( 
		\bar{C}_{n,0}^{-1/2} \bar{H}_{n,0}  \sqrt{n} (\hat{\theta}_n - \theta_{0,n}) 
		,
		- \bar{C}_{n,0}^{-1/2}
		\sqrt{n} \nabla M_n (\theta_{0,n})
		\right) 
		\notag
		\\
		& +
		\mathcal{W}_1
		\left(
		- \bar{C}_{n,0}^{-1/2}
		\sqrt{n} \nabla M_n (\theta_{0,n}) , Z 
		\right)
		\notag 
		\\
		=:  & W_1 + W_2.
	\end{align}

Observe first that
\begin{equation} \label{eq:w2:equals}
W_2 = 
\mathcal{W}_1
\left(
- \bar{C}_{n,0}^{-1/2}
\sqrt{n} \nabla M_n (\theta_{0,n}) , Z 
\right)
=
\mathcal{W}_1
\left(
- \bar{C}_{n,0}^{-1/2}
\sqrt{n} \nabla M_n (\theta_{0,n}) , -Z 
\right)
=
\mathcal{W}_1
\left(
 \bar{C}_{n,0}^{-1/2}
\sqrt{n} \nabla M_n (\theta_{0,n}) , Z 
\right).
\end{equation}
Hence, it is sufficient to bound $W_1 = \mathcal{W}_1
\left( 
\bar{C}_{n,0}^{-1/2} \bar{H}_{n,0}  \sqrt{n} (\hat{\theta}_n - \theta_{0,n}) 
,
- \bar{C}_{n,0}^{-1/2}
\sqrt{n} \nabla M_n (\theta_{0,n})
\right)$, which we now do. 
We have
\begin{align} \label{eq:using:lipschitzness}
	W_1
	= &
	\sup_{f \in \mathcal{L}_1} 
	\left|
	\mathbb{E} \left(
	f \left(
	\bar{C}_{n,0}^{-1/2} \bar{H}_{n,0} \sqrt{n}  (\hat{\theta}_n - \theta_{0,n})
	\right)
	\right)
	-
	\mathbb{E} \left(  f \left(
	-  \bar{C}_{n,0}^{-1/2} \sqrt{n} \nabla M_n (\theta_{0,n})
	\right)
	\right)
	\right| \notag
	\\
	\leq &
	\sup_{f \in \mathcal{L}_1} 
	\mathbb{E} \left(
	\left|
	f \left(
	\bar{C}_{n,0}^{-1/2} \bar{H}_{n,0} \sqrt{n}  (\hat{\theta}_n - \theta_{0,n})
	\right)
	-
	f \left(
	- \bar{C}_{n,0}^{-1/2} \sqrt{n} \nabla M_n (\theta_{0,n})
	\right)
	\right|
	\right) \notag
	\\
	& \leq 
	\mathbb{E} \left(
	\left| \left|
	\bar{C}_{n,0}^{-1/2} \bar{H}_{n,0} \sqrt{n}  (\hat{\theta}_n - \theta_{0,n})
	+  \bar{C}_{n,0}^{-1/2} \sqrt{n} \nabla M_n (\theta_{0,n})
	\right|
	\right|
	\right).
\end{align}
With $c_{\theta_0}$ as in Condition \ref{cond:global:identifiability},
observe that if $\hat{\theta}_n \in B( \theta_{0,n} , c_{\theta,0} )$ then $\nabla M_n( \hat{\theta}_n ) = 0$.
Hence, applying H\"older inequality, we obtain, 
\begin{align} \label{eq:introducing:Woneone:Wonetwo}
	W_1
	\leq &
	\mathbb{E} 
	\left(
	\left| \left|
	\bar{C}_{n,0}^{-1/2} \bar{H}_{n,0} \sqrt{n}  (\hat{\theta}_n - \theta_{0,n})
	+  \bar{C}_{n,0}^{-1/2} \sqrt{n} \nabla M_n (\theta_{0,n})
	\right|
	\right|^2
	\right)^{1/2}
	\mathbb{P} \left(  \hat{\theta}_n \not \in B( \theta_{0,n} , c_{\theta,0} ) \right)^{1/2}
	\notag \\
	& + 
	\mathbb{E} \left(
	\mathds{1}_{\{  \nabla M_n( \hat{\theta}_n ) = 0 \}}
		\mathds{1}_{\{  \hat{\theta}_n \in B( \theta_{0,n} , c_{\theta,0} ) \}}
	\left| \left|
	\bar{C}_{n,0}^{-1/2} \bar{H}_{n,0} \sqrt{n}  (\hat{\theta}_n - \theta_{0,n})
	+  \bar{C}_{n,0}^{-1/2} \sqrt{n} \nabla M_n (\theta_{0,n})
	\right|
	\right|
	\right) \notag \\
	= &  \mathbb{E}(W_{1,1})^{1/2} \mathbb{P}(A_{1,1})^{1/2}+ \mathbb{E}(W_{1,2}),
\end{align}
where we define
\[
W_{1,1} = 	
\left| \left|
\bar{C}_{n,0}^{-1/2} \bar{H}_{n,0} \sqrt{n}  (\hat{\theta}_n - \theta_{0,n})
+  \bar{C}_{n,0}^{-1/2} \sqrt{n} \nabla M_n (\theta_{0,n})
\right|
\right|^2, \hspace{3mm}
A_{1,1} = \left\{   \hat{\theta}_n \not \in B( \theta_{0,n} , c_{\theta,0})   \right\}
\]
and
\[
W_{1,2} = 
	\mathds{1}_{\{  \nabla M_n( \hat{\theta}_n ) = 0 \}}
	\mathds{1}_{\{  \hat{\theta}_n \in B( \theta_{0,n} , c_{\theta,0} ) \}}
\left| \left|
\bar{C}_{n,0}^{-1/2} \bar{H}_{n,0} \sqrt{n}  (\hat{\theta}_n - \theta_{0,n})
+  \bar{C}_{n,0}^{-1/2} \sqrt{n} \nabla M_n (\theta_{0,n})
\right|
\right|.
\]
Let us first bound $\mathbb{E}(W_{1,1})^{1/2} \mathbb{P}(A_{1,1})^{1/2}$. In $W_{1,1}$, $\bar{C}_{n,0}^{-1/2}$ is bounded from Condition \ref{cond:smallest:eigenvalue:covariance:nabla} and $\bar{H}_{n,0}$ is bounded from Condition \ref{cond:moment}. Furthermore, $\sqrt{n} \nabla M_n (\theta_{0,n})$ has mean zero from Condition \ref{cond:global:identifiability} and has bounded covariance matrix from Condition \ref{cond:var}.  Hence, since $\Theta$ is compact, with constants $0 < C_1 < \infty$ and $N_1 \in \IN$, we have for $n \geq N_1$,
\begin{equation*}
	\mathbb{E}(W_{1,1})^{1/2}
	\leq 
	C_1 \sqrt{n}.
\end{equation*}
Then Lemma \ref{lemma:proba:estimation:error} directly provides, for some constant $0<c_2<\infty$, $0 < C_2 < \infty$ and $N_2 \in \IN$, for $n \geq N_2$, 
\begin{equation*} 
\mathbb{P}(A_{1,1})^{1/2}
\leq 
C_2
\sqrt{n}
\exp(- c_2  n^{1/4}   ).
\end{equation*}

Hence, eventually, for some constants $0 < c_3 < \infty$, $0 < C_3 < \infty$ and $N_3 \in \IN$, for $n \geq N_3$,
	\begin{align} \label{eq:bound:Woneone}
		\mathbb{E}(W_{1,1})^{1/2} \mathbb{P}(A_{1,1})^{1/2}
		\leq 
		C_{3}
		n
		\exp(- c_{3} n^{1/4}  ).
	\end{align}
	
	Let us now bound $\mathbb{E}(W_{1,2})$. 
	When $\nabla M_n( \hat{\theta}_n ) = 0$ and $\hat{\theta}_n \in B( \theta_{0,n} , c_{\theta,0} )$, we have, since $ B( \theta_{0,n} , c_{\theta,0} ) \subset \ITheta$,
	\[
	0
	=
	\nabla M_n(\theta_{0,n})
	+
	\nabla^2 M_n( \tilde{\theta}_1, \ldots , \tilde{\theta}_p )
	(\hat{\theta}_n - \theta_{0,n}),
	\]
	where $\tilde{\theta}_1 , \ldots , \tilde{\theta}_p$ are on the segment between $\hat{\theta}_n$ and $\theta_{0,n}$ and where $\nabla^2 M_n( \tilde{\theta}_1, \ldots , \tilde{\theta}_p )$ is $p \times p$ with line $k$ equal to the line $k$ of  $\nabla^2 M_n( \tilde{\theta}_k )$ for $k \in \{1 , \ldots , p \}$. This yields, when $\nabla M_n( \hat{\theta}_n ) = 0$ and $\hat{\theta}_n \in B( \theta_{0,n} , c_{\theta,0} )$,
	\begin{equation} \label{eq:Wonetwo:equals}
		\bar{H}_{n,0}  \sqrt{n} (\hat{\theta}_n - \theta_{0,n})
		+   \sqrt{n} \nabla M_n (\theta_{0,n})
		=
		\sqrt{n} 
		\left(
		\mathbb{E}(\nabla^2 M_n( \theta_{0,n} ))
		-
		\nabla^2 M_n( \tilde{\theta}_1, \ldots , \tilde{\theta}_p )
		\right)
		(\hat{\theta}_n - \theta_{0,n}).
	\end{equation}
	Using Condition \ref{cond:smallest:eigenvalue:covariance:nabla}, we obtain, when $\nabla M_n( \hat{\theta}_n ) = 0$ and $\hat{\theta}_n \in B( \theta_{0,n} , c_{\theta,0} )$, for $n \geq N_{\theta_{0},\nabla}$,
	\begin{align*} 
		& W_{1,2}
		\notag
		\\
		& \leq
        \frac{1}{\sqrt{c_{\theta_0,\nabla}}}
		\sqrt{n} 
		\left| \left|
		\left(
		\mathbb{E}(\nabla^2 M_n( \theta_{0,n} ))
		-
		\nabla^2 M_n( \tilde{\theta}_1, \ldots , \tilde{\theta}_p )
		\right)
		(\hat{\theta}_n - \theta_{0,n})
		\right|
		\right|
		\notag
		\\
		& \leq 
	 \frac{1}{\sqrt{c_{\theta_0,\nabla}}}
		\sqrt{n} 
		\rho_1
		\left( 
		\mathbb{E}(\nabla^2 M_n( \theta_{0,n} ))
		-
		\nabla^2 M_n( \tilde{\theta}_1, \ldots , \tilde{\theta}_p ) 
		\right)
		|| \hat{\theta}_n - \theta_{0,n} ||
		\\
		& \leq 
		C_{4}
		\sqrt{n}
		\max_{j,k=1}^p
		\left|
		\mathbb{E}(\nabla^2 M_n( \theta_{0,n} ))_{j,k}
		-
		\nabla^2 M_n( \theta_{0,n} )_{j,k}
		\right|
		|| \hat{\theta}_n - \theta_{0,n} ||
		\\
		& 
		+
		C_4
		\sqrt{n}
		\max_{j,k,\ell=1}^p
		\sup_{\theta \in \ITheta}
		\left|
		\frac{\partial^3 M_n(\theta)}{ \partial \theta_j  \partial \theta_k \partial \theta_\ell }
		\right|
		|| \hat{\theta}_n - \theta_{0,n} ||^2,
	\end{align*}
where, in the last inequality, $0 < C_4 < \infty$ is a constant and we have used the mean value theorem. 
Using H\"older inequality together with Conditions \ref{cond:moment} and \ref{cond:var}, we obtain, for some constants $0 < C_5 < \infty$, $0 < C_6 < \infty$ and $N_5 \in \IN$, for $n \geq N_5$,
	\begin{align*}
		\mathbb{E}(W_{1,2})
		 \leq &
		C_{5} \sqrt{n} 
		\max_{j,k,=1}^p
		\mathrm{Var}( \nabla^2 M_n( \theta_{0,n} )_{j,k} )^{1/2}
		\mathbb{E} 
		\left( \mathds{1}_{ \{ \nabla M_n( \hat{\theta}_n ) = 0\} } 
	 \mathds{1}_{ \{ \hat{\theta}_n \in B( \theta_{0,n} , c_{\theta,0} ) \} } 
		|| \hat{\theta}_n - \theta_{0,n} ||^2 
		\right)^{1/2}
		\notag
		\\
		& +
		C_5
		\sqrt{n}
		\max_{j,k,\ell=1}^p
		\mathbb{E} \left(
\sup_{\theta \in \ITheta}
		\left|
		\frac{\partial^3 M_n(\theta)}{ \partial \theta_j \partial \theta_k \partial \theta_\ell }
		\right|^2
		\right)^{1/2}
		\mathbb{E}\left( 
		\mathds{1}_{ \{ \nabla M_n( \hat{\theta}_n ) = 0\} }
			 \mathds{1}_{ \{ \hat{\theta}_n \in B( \theta_{0,n} , c_{\theta,0} ) \} } 
		|| \hat{\theta}_n - \theta_{0,n} ||^4 
		\right)^{1/2}
		\notag
		\\
		 \leq &
		C_{6} \mathbb{E}
		\left( 
		\mathds{1}_{ \{ \nabla M_n( \hat{\theta}_n ) = 0\} }
			 \mathds{1}_{ \{ \hat{\theta}_n \in B( \theta_{0,n} , c_{\theta,0} ) \} } 
		|| \hat{\theta}_n - \theta_{0,n} ||^2
		\right)^{1/2}
		\\
& 		+
		C_{6} \sqrt{n}
		\mathbb{E} \left( 
		\mathds{1}_{ \{ \nabla M_n( \hat{\theta}_n ) = 0\} }
			 \mathds{1}_{ \{ \hat{\theta}_n \in B( \theta_{0,n} , c_{\theta,0} ) \} } 
		|| \hat{\theta}_n - \theta_{0,n} ||^4 
		\right)^{1/2}.
	\end{align*}
	
	We now apply Lemma \ref{lemma:proba:estimation:error:and:zero:gradient} with a constant $\gamma_1>0$ to be chosen later.  We obtain, with some constants $0 < c_7 < \infty$, $0 < c'_7 < \infty$, $0 < C_7 < \infty$ and $N_7 \in \IN$, for $n \geq N_7$ and $0 < t \leq c'_7$, 
	\begin{align} \label{eq:bound:proba:estimation:error}
	\mathbb{P}
\left(
\nabla M_n( \hat{\theta}_n ) = 0
,
||  \hat{\theta}_n - \theta_{0,n} ||
\geq t,
	\hat{\theta}_n \in B( \theta_{0,n} , c_{\theta,0} ) 
\right)
&\leq 
C_7
\frac{\log(n)^{ p \gamma_1 }}{t^p} 
\exp( - n c_7 t^2  )
\\
& +
C_7  n
\exp( - c_7 (\log n)^{\gamma_1}  )
+ C_7  n \exp( - c_7 n^{1/4} ). \notag
	\end{align}
	Hence, using $\mathbb{E}(X) \leq A + X_{\max} \mathbb{P}(X \geq A) $ for a non-negative random variable $X$ bounded by $X_{\max} >0$ and for $A>0$, we obtain, for a constant $0 < \gamma_2 < \infty$ to be chosen later, for a constant $0 < C_8 < \infty$, for $n \geq N_7$,
	\begin{align*}
		\mathbb{E}( W_{1,2} )
		\leq &
		\sqrt{ 
			C_8\frac{ \log(n)^{2 \gamma_2} }{ n }
			+
			C_8
			\mathbb{P}
			\left(
			\nabla M_n( \hat{\theta}_n ) = 0
			,
			||  \hat{\theta}_n  -  \theta_{0,n} ||
			\geq 
			\frac{ \log(n)^{\gamma_2} }{ \sqrt{n} }
			,	\hat{\theta}_n \in B( \theta_{0,n} , c_{\theta,0} ) 
			\right)
		}
		\\
		& +
		\sqrt{ 
			C_{8} n \frac{ \log(n)^{4 \gamma_2} }{ n^2 }
			+
			C_{8}  n
			\mathbb{P}
			\left(
			\nabla M_n( \hat{\theta}_n ) = 0
			,
			||  \hat{\theta}_n  -  \theta_{0,n} ||
			\geq 
			\frac{ \log(n)^{\gamma_2} }{ \sqrt{n} }
			,	\hat{\theta}_n \in B( \theta_{0,n} , c_{\theta,0} ) 
			\right)
		}.
	\end{align*}
	
	Hence from \eqref{eq:bound:proba:estimation:error},  for a constant $N_{8} \in \IN$ that may depend on $\gamma_1$ and $\gamma_2$, for $n  \geq N_8$,	
	\begin{align} \label{eq:bound:Wonetwo}
		& \mathbb{E}(W_{1,2})
		\leq \notag \\
		&
		\sqrt{C_8}
		\sqrt{
			\frac{(\log n)^{2 \gamma_2}}{ n }
			+
			C_7 
			\log(n)^{p \gamma_1}
			\frac{n^{p/2} 	\exp( -  c_7 (\log n)^{2 \gamma_2}  )}{\log(n)^{p \gamma_2}}
			+
			C_7  n
			\exp( - c_7 (\log n)^{\gamma_1}  )
			+ C_7  n \exp( -c_7 n^{1/4})
		}
		\notag
		\\
		& + 
			\sqrt{C_8}  \notag \\
			&
		\sqrt{
			\frac{(\log n)^{4 \gamma_2}}{ n }
			+
			C_7 
			\log(n)^{p \gamma_1}
			\frac{n^{p/2+1} \exp( -  c_7 (\log n)^{2 \gamma_2}  )}{\log(n)^{p \gamma_2}}
			+
			C_7 n^{2}
			\exp( - c_7 (\log n)^{\gamma_1}  )
			+ C_7  n^2 \exp( -c_7 n^{1/4})
		}.
	\end{align}

	Hence from \eqref{eq:basic:bound:Wasserstein}, \eqref{eq:w2:equals}, \eqref{eq:introducing:Woneone:Wonetwo}, \eqref{eq:bound:Woneone} and \eqref{eq:bound:Wonetwo}, choosing $\gamma_1$ and $\gamma_2$ as large enough constants, we obtain, for constants
	 $0 < \gamma_3 < \infty$, $0 < C_9 < \infty$, $N_9 \in \IN$, for $n \geq N_9$,
	\begin{equation*}
		 \mathcal{W}_1 \left(
		\bar{C}_{n,0}^{-1/2} \bar{H}_{n,0}
		\sqrt{n}  (\hat{\theta}_n - \theta_{0,n}) 
		, 
		Z  
		\right) \leq 
		\mathcal{W}_1 \left(   \bar{C}_{n,0}^{-1/2}
		\sqrt{n} \nabla M_n (\theta_{0,n}) , Z  \right)
		+C_9 \frac{(\log n)^{\gamma_3}}{ \sqrt{n} } . 
	\end{equation*}
This concludes the proof.
	
\end{proof}

\section{Proofs for Section \ref{subsection:logistic}}

\begin{proof}[Proof of Theorem \ref{theorem:logistic:bound}]
As stated previously, the function $M_n$ is given by
$$M_n(\theta) = \frac{1}{n}\sum_{i=1}^n \left( -y_ix_i^T\theta + \log(1 + \exp(x_i^T\theta)) \right)$$
where the $y_i$ are independent random variables with values in $\{0,1\}$. Defining $X_i = (x_i , y_i)$, we are in the framework of Theorem \ref{thm_classical_m_bnd}, so let us check that the required conditions indeed hold. 

It can be checked that there is a constant $0 < C_1 < \infty$ such that for any $Y$ as in \eqref{eq:Y:subgaussian}, $Y$ is almost surely bounded by $C_1$ (observe that $Y$ only takes two values). Hence the assumption \eqref{eq:Y:subgaussian} of sub-Gaussianity and bounded expectation holds.

Condition \ref{cond:Theta:covering} is already assumed to hold.
Condition \ref{cond:smoothness} can be shown simply.  
Let us show that Condition \ref{cond:global:identifiability} holds. Indeed, $ \nabla \Esp( M_n(\theta_0)) = 0$ can be seen directly from \eqref{eq:gradient:logistic}. Furthermore, from \eqref{eq:hessian:logistic}, we have, for $\theta \in \ITheta$,
\[
\nabla^2 \Esp( M_n(\theta) ) 
= 
\nabla^2  M_n(\theta)  
	 = 
\frac{1}{n}
\sum_{i=1}^n
\frac{e^{x_i^\top \theta}}{(1 + e^{x_i^\top \theta})^2} x_i x_i^\top.
\]
Hence, from Conditions \ref{cond:bounded:x:logistic} and \ref{cond:lambda:inf:cov:x:logistic}, there are constants $N_2 \in \IN$ and $0 < c_2 < \infty$ such that for $n \geq N_2$ and $\theta \in \ITheta$,
\begin{equation} \label{eq:lambda:inf:Hessian:logistic}
	\lambda_p \left( \nabla^2  M_n(\theta)   \right) \geq c_2. 
\end{equation} 
Hence, since $ \nabla \Esp( M_n(\theta_0)) = 0$, by strong convexity, Condition \ref{cond:global:identifiability} holds. 

Condition \ref{cond:local:identifiability} is a consequence of \eqref{eq:lambda:inf:Hessian:logistic}. Condition \ref{cond:smallest:eigenvalue:covariance:nabla} holds because $\Cov( \sqrt{n} \nabla M_n(\theta_0)) = \nabla^2 M_n(\theta_0)$ (this holds because we have a well-specified likelihood model and can also be checked directly).  

Finally, since all the quantities involved are uniformly bounded, option $1$ for checking Condition 4 holds. The second option could also be used instead, since the functions involved are all uniformly globally Lipschitz. 

	Hence Theorem \ref{thm_classical_m_bnd} can be applied, which concludes the proof. 
\end{proof}

\section{Proofs for Section \ref{subsection:cross:validation}}

\begin{lemma} \label{lemma:largest:eigenvalue:R}
		Assume that Conditions \ref{cond:minimal:distance} and \ref{cond:cross:validation:smoothness:and:bounds} hold. 
	There is a constant $C_R$ such that for $n \in \IN$,
		\begin{equation*}
		\sup_{\theta \in \Theta}
		\rho_1 \left( R_{n,\theta} \right)
		\leq 
		C_R.
	\end{equation*}
\end{lemma}

\begin{proof}[Proof of Lemma \ref{lemma:largest:eigenvalue:R}]
The lemma follows from \eqref{eq:bound:R} and from Lemma 4 in \cite{furrer2016asymptotic}.
\end{proof}

\begin{lemma} \label{lemma:gradient:CV}
	Assume that Conditions \ref{cond:minimal:distance} to \ref{cond:cross:validation:smoothness:and:bounds} hold. 
Then, we have, for $j \in \{1 , \ldots,p \}$, $\theta \in \ITheta$ and $n \in \IN$,
\begin{align} \label{eq:gradient:CV}
	(\nabla M_n(\theta))_{j}
	=
	\frac{1}{n}
	y^{(n)\top}
	B_{n,\theta,j}
	y^{(n)}
\end{align}
with
\begin{equation} \label{eq:matrice:for:gradient}
	B_{n,\theta,j}
	=
	2 R_{n,\theta}^{-1} \mathrm{diag}(R_{n,\theta}^{-1})^{-2} \left( \mathrm{diag} \left( R_{n,\theta}^{-1} \frac{\partial R_{n,\theta}}{\partial \theta_j} R_{n,\theta}^{-1} \right) \mathrm{diag}(R_{n,\theta}^{-1})^{-1} - R_{n,\theta}^{-1} \frac{\partial R_{n,\theta}}{\partial \theta_j} \right) R_{n,\theta}^{-1}.
\end{equation}

 For a constant and $0 < C_B < \infty$, we have, for $n \in \IN$,
\begin{equation} \label{eq:matrice:for:gradient:bound:singular:value}
	\max_{j=1,\ldots,p}
	\sup_{\theta \in \ITheta}
	\rho_1 \left( B_{n,\theta,j} \right)
	\leq 
  C_B.
\end{equation}
\end{lemma}

\begin{proof}[Proof of Lemma \ref{lemma:gradient:CV}]
The equation \eqref{eq:gradient:CV} is proved in \cite{bachoc14asymptotic,bachoc2020asymptotic}. 
The equation \eqref{eq:matrice:for:gradient:bound:singular:value} follows from Condition \ref{cond:smallest:eigen:value}, Lemma \ref{lemma:largest:eigenvalue:R} and \eqref{eq:bound:R:derivatives} and from the arguments in the proof of Proposition D.7 in \cite{bachoc14asymptotic}.
\end{proof}

\begin{lemma} \label{lemma:Hessian:CV}
		Assume that Conditions \ref{cond:minimal:distance} to \ref{cond:cross:validation:smoothness:and:bounds} hold. 
Then, we have, for $j,k \in \{1 , \ldots,p \}$, for $\theta \in \ITheta$, for $n \in \IN$,
\begin{align} \label{eq:hessian:CV}
	(\nabla^2 M_n(\theta))_{j,k}
	=
	\frac{1}{n}
	y^{(n)\top}
	C_{n,\theta,j,k}
	y^{(n)},
\end{align}
where the matrices $C_{n,\theta,j,k}$ satisfy, for a constant $0 < C_C < \infty$, for $n \in \IN$,
\begin{equation} \label{eq:matrice:for:hessian:bound:singular:value}
	\max_{j,k=1,\ldots,p}
	\sup_{\theta \in \ITheta}
	\rho_1 \left( C_{n,\theta,j,k} \right)
	\leq 
	C_C.
\end{equation}

\end{lemma}

\begin{proof}[Proof of Lemma \ref{lemma:Hessian:CV}]
Equation \eqref{eq:hessian:CV} is shown in \cite{bachoc14asymptotic},
where the matrices $C_{n,\theta,j,k}$ are obtained from the matrices $R_{n,\theta}$, $R_{n,\theta}^{-1}$, $\partial R_{n,\theta} / \partial \theta_j$, $\partial R_{n,\theta} / \partial \theta_k$ and 
	$
	\partial^2 R_{n,\theta} / \partial \theta_k \partial \theta_j = \partial^2 R_{n,\theta} / \partial \theta_j \partial \theta_k
	$,
	from sums and products and from the $\mathrm{diag}$ operator. The precise expressions of the matrices $C_{n,\theta,j,k}$ can be found in \cite{bachoc14asymptotic}.
Equation
\eqref{eq:matrice:for:hessian:bound:singular:value}
is then shown similarly to \eqref{eq:matrice:for:gradient:bound:singular:value}. 
\end{proof}

\begin{lemma} \label{lemma:third:dev:CV}
	Assume that Conditions \ref{cond:minimal:distance} to \ref{cond:cross:validation:smoothness:and:bounds} hold. 
 Then, for $j,k,\ell \in \{1 , \ldots,p \}$, for $\theta \in \ITheta$, for $n \in \IN$, we have
\begin{align} \label{eq:third:derivatives:CV}
	\frac{\partial^3 M_n(\theta)}{ \partial \theta_j \partial \theta_k \partial \theta_\ell }
	=
	\frac{1}{n}
	y^{(n)\top}
	D_{n,\theta,j,k,\ell}
	y^{(n)},
\end{align}
where the matrices $D_{n,\theta,j,k,\ell}$ satisfy, for some constant $0 < C_D < \infty$, for $n \in \IN$,
\begin{equation} \label{eq:matrice:for:third:der:bound:singular:value}
	\max_{j,k,\ell=1,\ldots,p}
	\sup_{\theta \in \ITheta}
	\rho_1 \left( D_{n,\theta,j,k,\ell} \right)
	\leq 
	C_D.
\end{equation}
\end{lemma}

\begin{proof}[Proof of Lemma \ref{lemma:third:dev:CV}]
The proof is the same as for Lemma  \ref{lemma:Hessian:CV}.
\end{proof}

\begin{lemma} \label{lemma:bounds:derivatives:CV}
	Assume that Conditions \ref{cond:minimal:distance} to \ref{cond:cross:validation:smoothness:and:bounds} hold. 
	Then, there is a constant $0 < C_{\partial,y} < \infty$ such that for $n \in \IN$,  
	\begin{equation} \label{eq:sup:gradient:CV}
		\sup_{\theta \in \ITheta}
		|| \nabla M_n( \theta ) ||
		\leq 
		C_{\partial,y}  \frac{1}{n}  ||y^{(n)}||^2,
	\end{equation}
	\begin{equation} \label{eq:sup:hessian:CV}
		\sup_{\theta \in \ITheta}
		\rho_1( \nabla^2 M_n( \theta ) )
		\leq 
	   C_{\partial,y}  \frac{1}{n}  ||y^{(n)}||^2
	\end{equation}
and
\begin{equation} \label{eq:sup:third:der:CV}
	\sup_{\theta \in \ITheta}
	\max_{j,k,\ell=1,\ldots,p}
	\left|
	\frac{\partial^3 M_n(\theta)}{ \partial \theta_j \partial \theta_k \partial \theta_\ell }
	\right|
	\leq 
	C_{\partial,y}  \frac{1}{n}  ||y^{(n)}||^2.
\end{equation}
\end{lemma}

\begin{proof}[Proof of Lemma \ref{lemma:bounds:derivatives:CV}]
Equations \eqref{eq:sup:gradient:CV}, \eqref{eq:sup:hessian:CV} and \eqref{eq:sup:third:der:CV}  follow from Lemmas \ref{lemma:gradient:CV}, \ref{lemma:Hessian:CV} and \ref{lemma:third:dev:CV}.
\end{proof}

\begin{lemma} \label{lemma:lambda:inf:Hessian:CV}
	Assume that Conditions \ref{cond:minimal:distance}, \ref{cond:smallest:eigen:value}, \ref{cond:cross:validation:smoothness:and:bounds} and \ref{cond:cross:validation:local:identifiability} hold. 
Then, Condition \ref{cond:local:identifiability} holds with $M_n$ as in \eqref{eq:Mntheta:CV} and $\theta_{0,n} = \theta_0$ as after \eqref{eq:Mntheta:CV}.
\end{lemma}

\begin{proof}[Proof of Lemma \ref{lemma:lambda:inf:Hessian:CV}]
Let $\alpha, \beta \in \mathbb{R}^p$ with $\alpha_1^2 + \cdots + \alpha_p^2 = 1$ and $\beta_1^2 + \cdots +  \beta_p^2 = 1$. For a matrix $M$, let $||M||_F$ be its Frobenius norm. We have
\begin{align} \label{eq:bound:derivative:frobenius}
	\frac{1}{\sqrt{n}} \left| \left|
	\sum_{\ell=1}^p
	\alpha_{\ell}
	\frac{\partial R_{n,\theta_0}}{ \partial \theta_{\ell} }
	-
	\sum_{\ell=1}^p
	\beta_{\ell}
	\frac{\partial R_{n,\theta_0}}{ \partial \theta_{\ell} }
	\right|
	\right|_F
	\leq
	|| \alpha - \beta || 
	\frac{1}{\sqrt{n}}
	\sum_{\ell=1}^p
	\left| \left|
	\frac{\partial R_{n,\theta_0}}{ \partial \theta_{\ell} }
	\right|
	\right|_F
	\leq C_{1} 
	|| \alpha - \beta ||,
\end{align}
with a constant $0 < C_1 < \infty$, from \eqref{eq:bound:R:derivatives} and Lemma 4 in \cite{furrer2016asymptotic}.
Hence, Condition \ref{cond:cross:validation:local:identifiability} implies that 
\begin{equation} \label{eq:inf:sum:derivatives:R}
	\liminf_{n  \to \infty}
	\inf_{\substack{\alpha_1,\ldots,\alpha_p \in \mathbb{R} \\ \alpha_1^2 + \dots + \alpha_p^2 = 1}}
	\frac{1}{n}
	\sum_{i,j=1}^n
	\left(
	\sum_{\ell=1}^p
	\alpha_\ell
	\frac{ \partial  (R_{n,\theta_0})_{i,j}}{ \partial \theta_\ell}
	\right)^2
	>0.
\end{equation}
The inequality \eqref{eq:inf:sum:derivatives:R} follows from \eqref{eq:bound:derivative:frobenius} and Condition \ref{cond:cross:validation:local:identifiability}. Indeed, if \eqref{eq:inf:sum:derivatives:R} does not hold we can consider a convergent subsequence of unit norm vectors of $\mathbb{R}^p$, $(\alpha_n)_{n \in \mathbb{N}}$, for which the quantity in \eqref{eq:inf:sum:derivatives:R} goes to zero. Considering the limit of $\alpha_n$ and \eqref{eq:bound:derivative:frobenius} yields a contradiction to Condition \ref{cond:cross:validation:local:identifiability}.

We have from the proof of Proposition 3.7 in \cite{bachoc14asymptotic} that there exists a constant $0 < c_2 < \infty$ such that, for all $\alpha \in \mathbb{R}^p$ with $\alpha_1^2 + \dots + \alpha_p^2 = 1$,
\[
\sum_{k,\ell=1}^p
\alpha_k \alpha_{\ell}
( \mathbb{E}( \nabla^2 M_n(\theta_0)))_{k,\ell}
\geq 
c_2
\frac{1}{n}
\sum_{i,j=1}^n
\left(
\sum_{\ell=1}^p
\alpha_\ell
\frac{ \partial  (R_{n,\theta_0})_{i,j}}{ \partial \theta_\ell}
\right)^2.
\]
Hence from \eqref{eq:inf:sum:derivatives:R} we obtain
\[
\liminf_{n \to \infty}
\lambda_{p}
\left(
\mathbb{E}( \nabla^2 M_n(\theta_0))
\right)
>0.
\]
\end{proof}

\begin{lemma} \label{lemma:lambda:inf:cov:CV}
		Assume that Conditions  \ref{cond:minimal:distance}, \ref{cond:smallest:eigen:value}, \ref{cond:cross:validation:smoothness:and:bounds} and \ref{cond:cross:validation:local:identifiability} hold. 
Then, condition \ref{cond:smallest:eigenvalue:covariance:nabla} holds with $M_n$ as in \eqref{eq:Mntheta:CV} and $\theta_{0,n} = \theta_0$ as after \eqref{eq:Mntheta:CV}.
\end{lemma}

\begin{proof}[Proof of Lemma \ref{lemma:lambda:inf:cov:CV}]
	Assume that for all constants $0 < c_1 < \infty$ and $N_1 \in \IN$, there is $n \geq N_{1}$ such that,
	\begin{equation} \label{eq:smallest:eigenvalue:cov:gradient}
		\lambda_p( \mathrm{Cov}( \sqrt{n} \nabla M_n(\theta_0) ) )
		\leq c_1.
	\end{equation}
	
	 Then, up to extracting a subsequence, there exists a sequence of unit vectors $(v_n)_{n \in \mathbb{N}}$ of $\mathbb{R}^p$ such that
	\begin{equation} \label{eq:vn:transp:cov:gradient:vn}
		v_n^\top \mathrm{Cov}( \sqrt{n} \nabla M_n(\theta_0) ) v_n 
		\to_{n \to \infty} 0.
	\end{equation}
	Let, for $t \geq 0$ such that $\theta_0 + t v_n \in \ITheta$,
	\[
	M_n(t) = M_n( \theta_0 + tv_n )
	\]
	and let $M_n'(t)$ be the derivative at $t$ of $t \mapsto M_n(t)$. We have
	\[
	M_n'(0) = \nabla M_n(\theta_0)^\top v_n.
	\]
	Hence \eqref{eq:vn:transp:cov:gradient:vn} implies
	\begin{equation} \label{eq:the:variance:goint:to:zero}
		\mathrm{Var}( \sqrt{n} M'_{n}(0) ) 
		\to_{n \to \infty} 0.
	\end{equation}
	Consider the logarithm of the likelihood
	\[
	L_n(t) = - \frac{1}{2} \log(\det( R_{n,t} ))
	- \frac{1}{2} y^{(n)\top} R_{n,t}^{-1} y^{(n)},
	\]
	where $R_{n,t} = R_{n , \theta_0 + t v_n}$. Let $K>0$ be fixed, to be selected later.
	Then, with $L_n'(t)$ and $L_n''(t)$ the first and second derivative of $t \mapsto L_n(t)$ at $t$, for $n$ such that $B( \theta_0 , K/\sqrt{n} ) \subset \ITheta$,
	\begin{align} \label{eq:log:likelihood:difference}
		\left|
		L_n(0) - L_n( K / \sqrt{n} )
		\right|
		& \leq 
		\frac{K}{\sqrt{n}}
		\sup_{ |t| \leq K / \sqrt{n}}
		|   L_n'(t) |
		\notag
		\\
		& \leq 
		\frac{K}{\sqrt{n}}
		|  L_n'(0) |
		+
		\left(\frac{K}{\sqrt{n}}\right)^2
		\sup_{ |t| \leq K / \sqrt{n}}
		| L_n''(t) |.
	\end{align}
	Let $\mathbb{P}_{n,t}$, $\mathbb{E}_{n,t}$ and $\mathrm{Var}_{n,t}$ be the Gaussian distribution of $y^{(n)}$, and the corresponding expectation and variance, assuming that $y^{(n)}$ has mean vector zero and covariance matrix $R_{n,t}$.
	From the arguments in \cite{bachoc14asymptotic}, $|L_n''(t)|$ is bounded by $n C_1 + C_1 ||y^{(n)}||^2$ and $L_n'(0)$ has expectation under $\mathbb{P}_{n,0}$ equal to zero and variance under $\mathbb{P}_{n,0}$ bounded by $C_1 n$, where $C_1$ can be chosen independently of $t \in [0,K]$.
	Hence the quantity in \eqref{eq:log:likelihood:difference} is bounded in $\mathbb{P}_{n,0}$ probability. We also have, for $n$ such that $B(\theta_0 , K / \sqrt{n}) \subset \ITheta$,
	\begin{align} \label{eq:log:likelihood:difference:two}
		\left|
		L_n(0) - L_n( K / \sqrt{n} )
		\right|
		& \leq 
		\frac{K}{\sqrt{n}}
		\sup_{ |t| \leq K / \sqrt{n}}
		|   L_n'(t) |
		\notag
		\\
		& \leq 
		\frac{K}{\sqrt{n}}
		|  L_n'(K / \sqrt{n}) |
		+
		2\left(\frac{K}{\sqrt{n}}\right)^2
		\sup_{ |t| \leq K / \sqrt{n}}
		| L_n''(t) |
	\end{align}
	and, similarly as before, the quantity in \eqref{eq:log:likelihood:difference:two} is bounded in $\mathbb{P}_{n,K/\sqrt{n}}$ probability. Hence, from Le Cam's first lemma (see for instance\cite[Lemma 6.4]{van2000asymptotic}), the measures $\mathbb{P}_{n,0}$ and $\mathbb{P}_{n,K/\sqrt{n}}$ are mutually contiguous.
	
	Now \eqref{eq:the:variance:goint:to:zero} and $\mathbb{E}_{n,0}( M_n'(0) ) = 0$ imply that
	\begin{equation} \label{eq:Enprime:to:zero}
		\sqrt{n} M'_{n}(0) 
		\to^{\mathbb{P}_{n,0}}_{n \to \infty} 0.
	\end{equation}
Hence, we have, again from Le Cam's first lemma and from \eqref{eq:Enprime:to:zero}, that
	\begin{equation} \label{eq:Eprime:large:to:be:contradicted}
		\sqrt{n} M'_{n}(0) 
		\to^{\mathbb{P}_{n,K/\sqrt{n}}}_{n \to \infty} 0.
	\end{equation}
We have, for $t \in [0,K/\sqrt{n}]$ and $n$ such that $B( \theta_0 , K/\sqrt{n} ) \in \ITheta$,
	\begin{align*}
		\left|
		\mathbb{E}_{n,0}( M_n''(0) )
		-
		\mathbb{E}_{n,t}( M_n''(t) )
		\right|
		\leq &
		\left|
		\mathbb{E}_{n,0}( M_n''(0) )
		-
		\mathbb{E}_{n,0}( M_n''(t) )
		\right|
		+
		\left|
		\mathbb{E}_{n,0}( M_n''(t) )
		-
		\mathbb{E}_{n,t}( M_n''(t) )
		\right|
		\\
		= &
		\left|
		\mathbb{E}_{n,0}( M_n''(0) )
		-
		\mathbb{E}_{n,0}( M_n''(t) )
		\right|
		+
		\frac{1}{n}
		\mathrm{Tr}
		\left(
		(R_{n,0} - R_{n,t})
		Q_{n,t}
		\right),
	\end{align*}
	with 
	\[
	Q_{n,t} = \sum_{j , k = 1}^p 
	(v_n)_j (v_n)_k
	C_{n,\theta_0+t v_n,j,k}
	\]
	from \eqref{eq:hessian:CV}. Hence from \eqref{eq:sup:third:der:CV}, \eqref{eq:matrice:for:hessian:bound:singular:value}, Cauchy-Schwarz inequality and Lemma \ref{lemma:largest:eigenvalue:R}, we have
	\[
	\sup_{t \in [0 , K/\sqrt{n}]}
	\left|
	\mathbb{E}_{n,0}( M_n''(0) )
	-
	\mathbb{E}_{n,t}( M_n''(t) )
	\right|
	\to_{n \to \infty} 0.
	\]
	Hence, from Lemma \ref{lemma:lambda:inf:Hessian:CV}, there exist $N_{2} \in \mathbb{N}$ and $0 < c_{2} < \infty$ such that, for $n \geq N_{2}$,
	\begin{equation} \label{eq:inf:Esp:Mn:second}
	\inf_{t \in [0 , K/\sqrt{n}]}
	\mathbb{E}_{n,t}( M_n''(t) )
	\geq 
	c_2.
	\end{equation}
	Note that $c_2$ can be chosen independently on $K$ while $N_{2}$ depends on $K$ (for instance, with $c_2 = c_{\theta_0,H}/2$ as in Condition \ref{cond:local:identifiability}).
	Similarly as for showing \eqref{eq:inf:Esp:Mn:second}, we can change the values of $c_2$ and $N_{2}$ such that, for $n \geq N_{2}$,
	\begin{equation} \label{eq:Esp:Ensecond:large}
		\inf_{t_1,t_2 \in [0 , K/\sqrt{n}]}
		\mathbb{E}_{n,t_1}( M_n''(t_2) )
		\geq 
	c_2.
	\end{equation}
	Again, $c_2$ can be chosen independently on $K$ while $N_{2}$ depends on $K$.
	Then, from the arguments of the proof of Lemma \ref{lemma:nabla:Mn:larger:than}, together with \eqref{eq:Esp:Ensecond:large}, we obtain, for $n$ larger than a constant $N_{K,1} \in \IN$,
	\[
	| \mathbb{E}_{n,K/\sqrt{n}}  \sqrt{n} M'_{n}(0)  |
	\geq 
	\sqrt{n}
	c_2
	\frac{K}{\sqrt{n}}.
	\]
	Furthermore, from \eqref{eq:gradient:CV}, \eqref{eq:matrice:for:gradient:bound:singular:value} and \eqref{eq:bound:R} we have, for $n$ larger than a constant $N_{K,2} \in \IN$, $\mathrm{Var}_{n,K/\sqrt{n}}( \sqrt{n} \linebreak[1] M'_{n}(0) ) \leq C_3 $ with a constant $0 < C_3 < \infty$ that does not depend on $K$. Hence, by taking $K$ large enough, the $\liminf$ of the $\mathbb{P}_{n,K/\sqrt{n}}$-probability that $|\sqrt{n} M'_{n}(0)|$ is larger than one can be made arbitrarily large. This is a contradiction to \eqref{eq:Eprime:large:to:be:contradicted}. Hence we have a contradiction to \eqref{eq:smallest:eigenvalue:cov:gradient}, which concludes the proof.
\end{proof}

\begin{proposition} \label{proposition:wasserstein:quadratic:form}
	Let $X = (Y^\top A_1Y,...,Y^\top A_pY)$ be a random vector, with $A_1,...,A_p$ symmetric $n \times n$ matrices, and $Y$ a Gaussian vector with covariance matrix $K$. Let $C$ be the $p \times p$ matrix with coefficients
	$$C_{i,j} = \Tr(KA_iKA_j)$$
	and $Z_C$ be a p-dimensional centered Gaussian vector with covariance matrix $C$. 
	Assume moreover that $X$ is centered, which is the same as assuming that
	$$\Tr(A_iK) = 0 \hspace{2mm}, i=1,\ldots,p.$$
	Then
	$$\mathcal{W}_1(X, Z_C) \leq
	\frac{ \sqrt{\lambda_1(C)} }{ \lambda_p(C) }
\sqrt{2\sum_{i,j = 1,...,p} \Tr((KA_iKA_j)^2)}.$$
\end{proposition}

\begin{proof}[Proof of Proposition \ref{proposition:wasserstein:quadratic:form}]
	The proposition is a direct consequence of \cite[Proposition 4.3]{nourdin2010clt}. 
\end{proof}

\begin{proof}[Proof of Theorem \ref{theorem:cross:validation:bound}]

Let us check that Conditions \ref{cond:Theta:covering} to \ref{cond:smallest:eigenvalue:covariance:nabla} hold in order to apply Theorem \ref{theorem:general:bound}. Condition \ref{cond:Theta:covering} is already assumed to hold. Condition \ref{cond:smoothness} holds because of Lemmas \ref{lemma:gradient:CV} to \ref{lemma:bounds:derivatives:CV}. Let us check the first part of Condition \ref{cond:concentration}. 
	From \eqref{eq:Mntheta:CV}, Condition \ref{cond:smallest:eigen:value},
	\eqref{eq:bound:R} and
	Lemma \ref{lemma:largest:eigenvalue:R} and as in \cite{bachoc14asymptotic}, we have
	\[
	M_n(\theta)
	=
	\frac{1}{n}
	y^{(n)\top}
	A_{n,\theta}
	y^{(n)}
	\]
	with $A_{n,\theta}$ symmetric and $\sup_{\theta \in \Theta} \rho_1( A_{n,\theta} ) \leq C_1$ for a constant $0 < C_1 < \infty$. By diagonalization, for each fixed $\theta \in \Theta$, there exist independent standard Gaussian variables $z_{n,\theta,1},\ldots,z_{n,\theta,n}$ and scalars $\lambda_{n,\theta,1},\ldots,\lambda_{n,\theta,n}$, such that, with a constant $0 < C_2 < \infty$,
	\[
	\sup_{n \in \mathbb{N}} \sup_{\theta \in \Theta} \max_{i=1}^n |\lambda_{n,\theta,i}| \leq C_{2}
\text{ and } 
	M_n(\theta)
	=
	\frac{1}{n}
	\sum_{i=1}^n
	\lambda_{n,\theta,i}
	z_{n,\theta,i}^2.
	\]
	Hence, we can apply Bernstein's inequality (for instance Theorem 2.8.1 in \cite{vershynin}) and we obtain, for $0 < \epsilon \leq 1$,
	\[
	\sup_{\theta \in \Theta} 
	\mathbb{P}
	( |M_n(\theta) - \mathbb{E}(M_n(\theta))  | \geq \epsilon )
	\leq C_3
	\exp( - n c_3 \epsilon^2 ),
	\]
	with constants $0 < c_3 < \infty$ and $0 < C_3 < \infty$ that do not depend on $\epsilon$. Hence the first part of Condition \ref{cond:concentration} indeed holds.
		The second part is shown in the same way, using Lemma \ref{lemma:gradient:CV}. 
		
	Let us check the first part of Condition \ref{cond:large:dev}. From \eqref{eq:sup:gradient:CV}, we obtain
	\[
	\mathbb{P} \left( \sup_{\theta \in \ITheta} 
	\left|  \left|
	\nabla M_n(\theta)
	\right| \right|
	\geq K
	\right)
	\leq 
	\mathbb{P} \left( 
C_{\delta,y}
	\max_{i=1}^n
	\left(
	y^{(n)}_i
	\right)^2
	\geq K
	\right)
	\leq 
	n \max_{i=1}^n 
	\mathbb{P}
	\left(
	C_{\delta,y} (y^{(n)}_i)^2 \geq K
	\right)
	\leq 
	C_{4} n \exp(- c_4 K  ),
	\]
	with constants $0 < c_4 < \infty$ and $0 < C_4 < \infty$, from, for instance, (A.2) in \cite{chatterjee2014superconcentration}. Hence the first part of Condition \ref{cond:large:dev} holds. The second part is shown similarly.

	Condition \ref{cond:moment}, \eqref{eq:cond:moments:un:deux} follows from \eqref{eq:sup:gradient:CV} and \eqref{eq:sup:hessian:CV}.
	Condition \ref{cond:moment}, \eqref{eq:cond:moments:trois} holds using first \eqref{eq:sup:third:der:CV}, then observing that from for instance (A.6) and (A.7) in \cite{paolella2018linear}, we have
	\[
	 \mathbb{E} \left( \left(  \frac{1}{n}  ||y^{(n)}||^2 \right)^2 \right) 
	 =
	  \frac{1}{n^2} \Tr \left(  R_{n,\theta_0}    \right)^2
	 +
	 \frac{2}{n^2} \Tr \left(  R_{n,\theta_0}^2   \right),
	\]
	and finally using Lemma \ref{lemma:largest:eigenvalue:R}.
 
The first part of Condition \ref{cond:var} is shown from Lemma \ref{lemma:gradient:CV} and, e.g., (A.7) in \cite{paolella2018linear}. The second part is shown similarly from Lemma \ref{lemma:Hessian:CV}. 
	In Condition \ref{cond:global:identifiability}, the offline equation follows from Condition \ref{cond:cross:validation:global:identifiability} and the proof of Proposition 3.4 in \cite{bachoc14asymptotic}. Furthermore, $\Esp( \nabla M_n (\theta_{0}) ) = 0$ is shown for instance in  \cite{bachoc14asymptotic} and can also be checked directly. Thus  Condition \ref{cond:global:identifiability} holds. 
 Condition \ref{cond:local:identifiability} holds from Lemma \ref{lemma:lambda:inf:Hessian:CV}. 	Condition \ref{cond:smallest:eigenvalue:covariance:nabla} holds from Lemma \ref{lemma:lambda:inf:cov:CV}.
	
	Hence Theorem \ref{theorem:general:bound} can be applied. From this theorem, in order to conclude the proof, it is sufficient to show that, with a constant $0 < C_5 < \infty$,
	\begin{equation} \label{eq:to:show:final:CV}
		\mathcal{W}_1
	\left( 
	\bar{C}_{n,0}^{-1/2}
	\sqrt{n} \nabla M_n (\theta_{0}) , Z 
	\right)
	\leq 
	\frac{C_5}{ \sqrt{n} }.
	\end{equation}
	
	The quantity $	\sqrt{n} \nabla M_n (\theta_{0})$ satisfies the condition of Proposition \ref{proposition:wasserstein:quadratic:form}, with $Y = y^{(n)}$ and, for $j=1,\ldots,p$,
	\[
	A_j = \frac{1}{\sqrt{n}} B_{n,\theta_0,j},
	\]
	from Lemma \ref{lemma:gradient:CV}. From Condition \ref{cond:global:identifiability}, then indeed $\Esp(  y^{(n)\top} A_j y^{(n)} )  = 0$. Then Proposition \ref{proposition:wasserstein:quadratic:form} yields
	\begin{equation} \label{eq:W1:nabla:Zn}
	\mathcal{W}_1( 	\sqrt{n} \nabla M_n (\theta_{0}) , Z_n )
	\leq 
	C_6,
	\end{equation}
	where $Z_n$ is a Gaussian vector with mean zero and covariance matrix $\bar{C}_{n,0}$, for a constant $0 < C_6 < \infty$, from \eqref{eq:matrice:for:gradient:bound:singular:value}, Conditions \ref{cond:var} and \ref{cond:smallest:eigenvalue:covariance:nabla} and Lemma \ref{lemma:largest:eigenvalue:R}.
	Then from Lemma \ref{lemma:Lipschitz:Wasserstein} and Condition \ref{cond:smallest:eigenvalue:covariance:nabla}, 
	\begin{align*}
		\mathcal{W}_1
\left( 
\bar{C}_{n,0}^{-1/2}
\sqrt{n} \nabla M_n (\theta_{0}) , Z 
\right)
			\leq 	\frac{C_6}{\sqrt{c_{\theta_0,\nabla} }}.
	\end{align*}
Hence, \eqref{eq:to:show:final:CV} is shown, which concludes the proof.
\end{proof}

\section*{Acknowledgments}
This work was supported by the Project MESA (ANR-18-CE40-006) of the French National Research Agency (ANR).


\begin{thebibliography}{10}
	
	\bibitem{anastasiou2017bounds}
	A.~Anastasiou.
	\newblock Bounds for the normal approximation of the maximum likelihood
	estimator from m-dependent random variables.
	\newblock {\em Statistics \& Probability Letters}, 129:171--181, 2017.
	
	\bibitem{anastasiou2018assessing}
	A.~Anastasiou.
	\newblock Assessing the multivariate normal approximation of the maximum
	likelihood estimator from high-dimensional, heterogeneous data.
	\newblock {\em Electronic Journal of Statistics}, 12(2):3794--3828, 2018.
	
	\bibitem{anastasiou2019normal}
	A.~Anastasiou, K.~Balasubramanian, and M.~A. Erdogdu.
	\newblock Normal approximation for stochastic gradient descent via
	non-asymptotic rates of martingale {CLT}.
	\newblock In {\em Conference on Learning Theory}, pages 115--137. PMLR, 2019.
	
	\bibitem{anastasiou2020multivariate}
	A.~Anastasiou and R.~E. Gaunt.
	\newblock Multivariate normal approximation of the maximum likelihood estimator
	via the delta method.
	\newblock {\em Brazilian Journal of Probability and Statistics},
	34(1):136--149, 2020.
	
	\bibitem{anastasiou2020wasserstein}
	A.~Anastasiou and R.~E. Gaunt.
	\newblock Wasserstein distance error bounds for the multivariate normal
	approximation of the maximum likelihood estimator.
	\newblock {\em arXiv preprint arXiv:2005.05208}, 2020.
	
	\bibitem{anastasiouLey2017bounds}
	A.~Anastasiou and C.~Ley.
	\newblock Bounds for the asymptotic normality of the maximum likelihood
	estimator using the delta method.
	\newblock {\em ALEA, Latin American Journal of Probability and Mathematical
		Statistics}, 14:153--171, 2017.
	
	\bibitem{anastasiouReinert2017bounds}
	A.~Anastasiou and G.~Reinert.
	\newblock Bounds for the normal approximation of the maximum likelihood
	estimator.
	\newblock {\em Bernoulli}, 23(1):191--218, 2017.
	
	\bibitem{anastasiou2020bounds}
	A.~Anastasiou and G.~Reinert.
	\newblock Bounds for the asymptotic distribution of the likelihood ratio.
	\newblock {\em The Annals of Applied Probability}, 30(2):608--643, 2020.
	
	\bibitem{Bachoc2013cross}
	F.~Bachoc.
	\newblock Cross validation and maximum likelihood estimations of
	hyper-parameters of {Gaussian} processes with model mispecification.
	\newblock {\em Computational Statistics and Data Analysis}, 66:55--69, 2013.
	
	\bibitem{bachoc14asymptotic}
	F.~Bachoc.
	\newblock Asymptotic analysis of the role of spatial sampling for covariance
	parameter estimation of {Gaussian} processes.
	\newblock {\em Journal of Multivariate Analysis}, 125:1--35, 2014.
	
	\bibitem{Bachoc2021asymptotic}
	F.~Bachoc.
	\newblock {\em Asymptotic Analysis of Maximum Likelihood Estimation of
		Covariance Parameters for Gaussian Processes: An Introduction with Proofs},
	pages 283--303.
	\newblock Springer International Publishing, Cham, 2021.
	
	\bibitem{bachoc2020asymptotic}
	F.~Bachoc, J.~B{\'e}tancourt, R.~Furrer, and T.~Klein.
	\newblock Asymptotic properties of the maximum likelihood and cross validation
	estimators for transformed {Gaussian} processes.
	\newblock {\em Electronic Journal of Statistics}, 14(1):1962--2008, 2020.
	
	\bibitem{bachoc2016smallest}
	F.~Bachoc and R.~Furrer.
	\newblock On the smallest eigenvalues of covariance matrices of multivariate
	spatial processes.
	\newblock {\em Stat}, 5(1):102--107, 2016.
	
	\bibitem{bachoc2020spatial}
	F.~Bachoc, M.~G. Genton, K.~Nordhausen, A.~Ruiz-Gazen, and J.~Virta.
	\newblock Spatial blind source separation.
	\newblock {\em Biometrika}, 107(3):627--646, 2020.
	
	\bibitem{bachoc2020uniformly}
	F.~Bachoc, D.~Preinerstorfer, and L.~Steinberger.
	\newblock Uniformly valid confidence intervals post-model-selection.
	\newblock {\em The Annals of Statistics}, 48(1):440--463, 2020.
	
	\bibitem{bentkus1997berry}
	V.~Bentkus, M.~Bloznelis, and F.~G{\"o}tze.
	\newblock A {Berry}--{Esseen} bound for {M}-estimators.
	\newblock {\em Scandinavian journal of statistics}, 24(4):485--502, 1997.
	
	\bibitem{berk2013valid}
	R.~Berk, L.~Brown, A.~Buja, K.~Zhang, and L.~Zhao.
	\newblock Valid post-selection inference.
	\newblock {\em The Annals of Statistics}, pages 802--837, 2013.
	
	\bibitem{bonis2020stein}
	T.~Bonis.
	\newblock Stein’s method for normal approximation in {Wasserstein} distances
	with application to the multivariate central limit theorem.
	\newblock {\em Probability Theory and Related Fields}, 178(3):827--860, 2020.
	
	\bibitem{boucheron2013concentration}
	S.~Boucheron, G.~Lugosi, and P.~Massart.
	\newblock {\em Concentration inequalities: A nonasymptotic theory of
		independence}.
	\newblock Oxford university press, 2013.
	
	\bibitem{casella2021statistical}
	G.~Casella and R.~L. Berger.
	\newblock {\em Statistical inference}.
	\newblock Pacific Grove, CA: Duxbury Press, 2021.
	
	\bibitem{chatterjee2009stein}
	S.~Chatterjee.
	\newblock A new method of normal approximation.
	\newblock {\em Annals of Probability}, 36(4):1584--1610, 2008.
	
	\bibitem{chatterjee2009poincare}
	S.~Chatterjee.
	\newblock Fluctuations of eigenvalues and second order {P}oincar\'{e}
	inequalities.
	\newblock {\em Probability Theory and Related Fields}, 143(1-2):1--40, 2009.
	
	\bibitem{chatterjee2014superconcentration}
	S.~Chatterjee.
	\newblock {\em Superconcentration and related topics}, volume~15.
	\newblock Springer, 2014.
	
	\bibitem{chiles2009geostatistics}
	J.-P. Chiles and P.~Delfiner.
	\newblock {\em Geostatistics: Modeling Spatial Uncertainty}.
	\newblock John Wiley \& Sons, 2009.
	
	\bibitem{cressie2015statistics}
	N.~Cressie.
	\newblock {\em Statistics for spatial data}.
	\newblock John Wiley \& Sons, 2015.
	
	\bibitem{decreusefond2019dirichlet}
	L.~Decreusefond and H.~Halconruy.
	\newblock Malliavin and {D}irichlet structures for independent random
	variables.
	\newblock {\em Stochastic Processes and their Applications}, 129(8):2611--2653,
	2019.
	
	\bibitem{dubrule83cross}
	O.~Dubrule.
	\newblock Cross validation of {Kriging} in a unique neighborhood.
	\newblock {\em Mathematical Geology}, 15:687--699, 1983.
	
	\bibitem{duerinckx2021glauber}
	M.~Duerinckx.
	\newblock On the size of chaos via {G}lauber calculus in the classical
	mean-field dynamics.
	\newblock {\em Communications in Mathematical Physics}, 382(1):613--653, 2021.
	
	\bibitem{Fahrmeir90}
	L.~Fahrmeir.
	\newblock Maximum likelihood estimation in misspecified generalized linear
	models.
	\newblock {\em Statistics}, 21(4):487--502, 1990.
	
	\bibitem{furrer2016asymptotic}
	R.~Furrer, F.~Bachoc, and J.~Du.
	\newblock Asymptotic properties of multivariate tapering for estimation and
	prediction.
	\newblock {\em Journal of Multivariate Analysis}, 149:177--191, 2016.
	
	\bibitem{genton2015cross}
	M.~G. Genton and W.~Kleiber.
	\newblock Cross-covariance functions for multivariate geostatistics.
	\newblock {\em Statistical Science}, 30(2):147--163, 2015.
	
	\bibitem{gozlan}
	N.~Gozlan.
	\newblock A characterization of dimension free concentration in terms of
	transportation inequalities.
	\newblock {\em The Annals of Probability}, 37(6):2480--2498, 2009.
	
	\bibitem{hallin2009local}
	M.~Hallin, Z.~Lu, and K.~Yu.
	\newblock Local linear spatial quantile regression.
	\newblock {\em Bernoulli}, 15(3):659--686, 2009.
	
	\bibitem{huber1967under}
	P.~J. Huber.
	\newblock The behavior of maximum likelihood estimates under nonstandard
	conditions.
	\newblock In {\em Proceedings of the Fifth Berkeley Symposium on Mathematical
		Statistics and Probability: Weather modification}, volume~5, page 221. Univ
	of California Press, 1967.
	
	\bibitem{lv14model}
	J.~Lv and J.~S. Liu.
	\newblock Model selection principles in misspecified models.
	\newblock {\em Journal of the Royal Statistical Society Series B}, 76:141--167,
	2014.
	
	\bibitem{mardia84maximum}
	K.~Mardia and R.~Marshall.
	\newblock Maximum likelihood estimation of models for residual covariance in
	spatial regression.
	\newblock {\em Biometrika}, 71:135--146, 1984.
	
	\bibitem{nourdin2020breuer}
	I.~Nourdin and D.~Nualart.
	\newblock The functional {B}reuer-{M}ajor theorem.
	\newblock {\em Probability Theory and Related Fields}, 176(1-2):203--218, 2020.
	
	\bibitem{nourdin2009stein}
	I.~Nourdin and G.~Peccati.
	\newblock Stein's method on {W}iener chaos.
	\newblock {\em Probability Theory and Related Fields}, 145(1-2):75--118, 2009.
	
	\bibitem{nourdin2012livre}
	I.~Nourdin and G.~Peccati.
	\newblock {\em Normal approximations with {M}alliavin calculus: From Stein's
		method to universality}, volume 192 of {\em Cambridge Tracts in Mathematics}.
	\newblock Cambridge University Press, Cambridge, 2012.
	
	\bibitem{nourdin2011breuer}
	I.~Nourdin, G.~Peccati, and M.~Podolskij.
	\newblock Quantitative {B}reuer-{M}ajor theorems.
	\newblock {\em Stochastic Processes and their Applications}, 121(4):793--812,
	2011.
	
	\bibitem{nourdin2009poincare}
	I.~Nourdin, G.~Peccati, and G.~Reinert.
	\newblock Second order {P}oincar\'{e} inequalities and {CLT}s on {W}iener
	space.
	\newblock {\em Journal of Functional Analysis}, 257(2):593--609, 2009.
	
	\bibitem{nourdin2010invariance}
	I.~Nourdin, G.~Peccati, and G.~Reinert.
	\newblock Invariance principles for homogeneous sums: universality of
	{G}aussian {W}iener chaos.
	\newblock {\em Annals of Probability}, 38(5):1947--1985, 2010.
	
	\bibitem{nourdin2010clt}
	I.~Nourdin, G.~Peccati, and A.~R\'{e}veillac.
	\newblock Multivariate normal approximation using {S}tein's method and
	{M}alliavin calculus.
	\newblock {\em Annales de l'Institut Henri Poincar\'{e} Probabilit{\'e}s et
		Statistiques}, 46(1):45--58, 2010.
	
	\bibitem{paolella2018linear}
	M.~S. Paolella.
	\newblock {\em Linear Models and Time-Series Analysis: Regression, ANOVA, ARMA
		and GARCH}.
	\newblock John Wiley \& Sons, 2018.
	
	\bibitem{pinelis2017optimal}
	I.~Pinelis.
	\newblock Optimal-order uniform and nonuniform bounds on the rate of
	convergence to normality for maximum likelihood estimators.
	\newblock {\em Electronic Journal of Statistics}, 11(1):1160--1179, 2017.
	
	\bibitem{pinelis2016optimal}
	I.~Pinelis and R.~Molzon.
	\newblock Optimal-order bounds on the rate of convergence to normality in the
	multivariate delta method.
	\newblock {\em Electronic Journal of Statistics}, 10(1):1001--1063, 2016.
	
	\bibitem{potscher2013dynamic}
	B.~M. P{\"o}tscher and I.~R. Prucha.
	\newblock {\em Dynamic nonlinear econometric models: Asymptotic theory}.
	\newblock Springer Science \& Business Media, 2013.
	
	\bibitem{pronzato2013design}
	L.~Pronzato and A.~P{\'a}zman.
	\newblock Design of experiments in nonlinear models.
	\newblock {\em Lecture notes in statistics}, 212:1, 2013.
	
	\bibitem{ross11}
	N.~Ross.
	\newblock Fundamentals of {S}tein's method.
	\newblock {\em Probability Surveys}, 8:210--293, 2011.
	
	\bibitem{shao2021berry}
	Q.-M. Shao and Z.-S. Zhang.
	\newblock {Berry--Esseen} bounds for multivariate nonlinear statistics with
	applications to {M}-estimators and stochastic gradient descent algorithms.
	\newblock {\em arXiv preprint arXiv:2102.04923}, 2021.
	
	\bibitem{van2000asymptotic}
	A.~W. Van~der Vaart.
	\newblock {\em Asymptotic statistics}, volume~3.
	\newblock Cambridge University Press, 2000.
	
	\bibitem{vershynin}
	R.~Vershynin.
	\newblock {\em High-dimensional probability: An introduction with applications
		in data science}, volume~47 of {\em Cambridge Series in Statistical and
		Probabilistic Mathematics}.
	\newblock Cambridge University Press, Cambridge, 2018.
	
	\bibitem{wackernagel2013multivariate}
	H.~Wackernagel.
	\newblock {\em Multivariate geostatistics: an introduction with applications}.
	\newblock Springer Science \& Business Media, 2013.
	
	\bibitem{white1982maximum}
	H.~White.
	\newblock Maximum likelihood estimation of misspecified models.
	\newblock {\em Econometrica: Journal of the econometric society}, 50(1):1--25,
	1982.
	
	\bibitem{zhang04inconsistent}
	H.~Zhang.
	\newblock Inconsistent estimation and asymptotically equivalent interpolations
	in model-based geostatistics.
	\newblock {\em Journal of the American Statistical Association}, 99:250--261,
	2004.
	
	\bibitem{Zhang2010kriging}
	H.~Zhang and Y.~Wang.
	\newblock Kriging and cross validation for massive spatial data.
	\newblock {\em Environmetrics}, 21:290--304, 2010.
	
\end{thebibliography}
\end{document}